\documentclass{amsart} 

\usepackage{graphicx}%
\usepackage{multirow}%
\usepackage{amsmath,amssymb,amsfonts}%
\usepackage{amsthm}%
\usepackage{mathrsfs}%
\usepackage[title]{appendix}%
\usepackage{xcolor}%
\usepackage{textcomp}%
\usepackage{manyfoot}%
\usepackage{booktabs}%
\usepackage{algorithm}%
\usepackage{algorithmicx}%
\usepackage{algpseudocode}%
\usepackage{listings}%

\usepackage[backref,bookmarks=true,colorlinks=true,pdfstartview=FitV,linkcolor=blue, citecolor=red,urlcolor=green]{hyperref}

\newtheorem{thm}{Theorem}[section]
\newtheorem{lem}[thm]{Lemma}
\newtheorem{cor}[thm]{Corollary}
\newtheorem{prop}[thm]{Proposition}

\theoremstyle{remark}
\newtheorem{rem}[thm]{Remark}

\theoremstyle{definition}
\newtheorem{defn}[thm]{Definition}
\newcommand\bR{{\mathbb{R}}}

\newcommand\bZ{{\mathbb Z}}
\newcommand\bQ{{\mathbb Q}}

\newcommand\Hom{{\rm Hom}}
\newcommand\dev{{\bf dev}}
\newcommand\SI{{\mathbb{S}}}

\newcommand\clo{{\rm Cl}}

\newcommand\ra{\rightarrow}

\newcommand\che{\check}
\newcommand\emp{\emptyset}
\newcommand\eps{\epsilon}

\newcommand\Aut{{\mathbf{Aut}}}
\newcommand\Idd{{\rm I}}

\newcommand\rpn{\mathbb{RP}^n}

\newcommand\SL{{\mathsf{SL}}}

\newcommand\GL{{\mathsf{GL}}}

\newcommand\Def{\mathrm{Def}}

\newcommand{\llrrV}[1]{
	\left|\mkern-2mu\left|#1\right|\mkern-2mu\right|}

\begin{document}
	

\title[real projective manifolds]{$\rpn \# \rpn$  and some others admit no real projective structure}


\author{Suhyoung Choi}

\thanks{This work was supported by 
the National Research Foundation of Korea(NRF) grant funded by the Korea government 
(MEST) (No. NRF-2022R1A2C3003162).}




\address{Department of Mathematics \\ KAIST \\Daejeon 34141, South Korea}



\date{\today}



\begin{abstract} 
A manifold $M$ possesses a {\em real projective structure} if it has an atlas consisting of charts mapping to $\SI^n$, where the transition maps lie in $\SL_\pm(n+1, \bR)$. In this context, we present a concise proof demonstrating that $\rpn\#\rpn$ and a few other manifolds do not possess a real projective structure when $n\geq3$. Notably, our proof is shorter than those provided by Cooper-Goldman for $n=3$ and \c{C}oban for $n\geq 4$. To do this, we 
reprove the classification of closed real projective manifolds with infinite-cyclic holonomy groups 
by Benoist due to a small error. 
We will leverage the concept of the octantizability of real projective manifolds with nilpotent holonomy groups, as introduced by Benoist and Smillie, which serves as a powerful tool.
\end{abstract} 


\maketitle

\section{Introduction} 

A real projective structure on a manifold is defined as a maximal atlas of charts mapping to $\SI^n$, where the transition maps belong to $\SL_\pm(n+1, \bR)$. For a comprehensive introduction, refer to \cite{Gnote, psconv}.


In this paper, we will show that $\rpn \# \rpn$ and some generalizations do not admit a real projective structure in Corollary \ref{cor:main}. 
We have a simpler proof than those given by Cooper-Goldman \cite{CG15} and 
\c{C}oban \cite{Coban21}, building on the technical work of Benoist \cite{BenTor}. 
(See also Ruffoni \cite{Ruffoni}.)
Let us assume that $n \geq 3$. In this context, we define a great sphere of dimension $0$ as a pair of antipodal points.

We need some terminology in this paper for an expository purpose.
This is also due to the fact that the theory here was studied by only a few people. 

\begin{defn} \label{defn:gHopf}
 {\em Generalized Hopf manifolds of index $(p, q)$ for $1\leq p \leq q$} are closed real projective manifolds with virtually cyclic holonomy groups that are projectively covered by a universal cover of $\SI^n \setminus S^p \setminus S^q$, where $p \leq q$. Here, $S^p$ and $S^q$ represent disjoint totally geodesic spheres of dimensions $p$ and $q$, respectively. It is required that $p + q = n-1, 1 \leq p \leq q$, and $p, q \in [1, n-2]$ with $n \geq 3$.
\end{defn} 
It is worth noting that unless 
$p = 1$ and $q = n-2$, $\SI^n \setminus S^p \setminus S^q$ itself serves as a universal cover.
(In Example 4.7 of \cite{BenTor}, these are denoted by $\mathbb{R}_g$.)


\begin{defn} \label{defn:pHopf} 
A {\em pure Hopf manifold} is defined when $p = 0$ and $q = n-1$, where $n \geq 2$. It represents a closed real projective manifold with a virtually cyclic holonomy group that is projectively covered by a component of $\SI^n \setminus S^0 \setminus S^{n-1}$. Here, $S^{n-1}$ denotes a great sphere of dimension $n-1$, while $S^0$ corresponds to a pair of antipodal points that are disjoint from $S^{n-1}$.
\end{defn} 
(For further details, refer to \cite{Hopf}.)





\begin{itemize} 
\item For a vector $x \in \SI^k$, we denote by $x_-$ its antipodal vector. 
We note that $\rpn \# \rpn$, $n \geq 3$, is diffeomorphic to 
to the quotient space of $\SI^1\times \SI^{n-1}$ by the $\bZ_2$-actions given by 
$(t, x) \ra (t_-, x_-)$ for $t\in \SI^1, x\in \SI^{n-1}$.
We call this a {\em Cooper-Goldman manifold}. 
\item A {\em generalized Cooper-Goldman} manifold of type 
$p, q, 2\leq p, q,$ is 
$(\SI^1 \times \SI^p \times \SI^q)/\bZ_2$ where $\bZ_2$ acts on 
$\SI^1 \times \SI^p \times \SI^q$ 
 by $(x, y, z) \ra (x_-, y_-, z_-)$ for $x\in \SI^1, y\in \SI^p, z\in \SI^q$. 
\end{itemize} 

 By a {\em nontrivially twisted
extension of an abelian group $G$} by another group $H$, we
mean a group $G'$ with an exact sequence 
\[ 1 \ra G \ra G' \ra H \ra 1\] 
where $H$ acts nontrivially on $G$.
Of course $G'$ will be nonabelian in general. 

By Corollary \ref{cor:main},
the Cooper-Goldman manifold, a generalized Cooper-Goldman 
manifold of type $p, q$ for 
$p\leq q, p+q = n-1, 2 \leq p \leq  q \leq n-3$, $n \geq 3$,  
do not admit real projective structures. 
However, there are real projective structures for some manifolds that do not satisfy the $\pi_1$-action condition, as we explain in Section \ref{subsec:counter}. 
Also, any manifold covered by these will not admit any real projective structure.

 \begin{cor} \label{cor:main} 
Let $M$ be a closed real projective manifold, $n \geq 3$,  where $\pi_1(M)$ acts trivially on 
$\pi_i(M) \otimes \bZ_2$ for $i \geq 1$. 
	Then the holonomy group and the fundamental group of
	cannot be a nontrivially twisted finite extension of $\bZ$. 
\end{cor} 
See Theorem 8 of Section 3 of Chapter 3 of \cite{Spanier} for the action 
of $\pi_1(M)$. This result is more general than the work of Cooper, Goldman, and 
\c{C}oban since their conditions are stronger than the above fundamental-group condition. 
But of course this is a field with a very few examples due to the classification 
by Proposition 4.9 of Benoist \cite{BenTor}. 



For proof, we will use Proposition 4.9 of Benoist \cite{BenTor}, which classifies
closed real projective manifolds with cyclic holonomy groups. 
However, the proof contained Lemma 4.11 which is incorrect if there is nonempty boundary. 
We explain this in Section \ref{subsec:counter}.  
Of course, the lemma itself is right if we assume that $M$ is a closed manifold with 
a cyclic holonomy by our results, and the overall philosophy is quite correct. 


In addition, the final part of the proof of Proposition 4.9 of Benoist \cite{BenTor} is 
very slim giving us ideas only. Also, there is no mention of the cases when the generators 
have nonreal eigenvalues. 
We needed to fill these in by Lemma \ref{lem:persist}. 

We assume that the holonomy is regular diagonalizable. Then Lemma 4.11 
of \cite{BenTor} is correct as proved in Example 5.10 of the same paper. 
There might be a way to directly do these; However, 
for general cases, we will use perturbations following the idea of Cooper and 
Goldman \cite{CG15}. 

We will reprove Proposition 4.9 in \cite{BenTor} by Proposition \ref{prop:mainp}
following from Lemma \ref{lem:persist}. 





For history and the prospects of this field, we note that Nagano-Yagi \cite{NY} classified affine structures on tori, and Goldman \cite{Gthesis} did the same for real projective structures on tori. Smillie \cite{Smillie2} studied the affine manifolds with diagonalizable holonomy, giving us some means to understand these. (See Nam \cite{Nam} also.) Choi and Goldman classified the real projective structures on closed surfaces using some of these ideas. (See \cite{CG93}).
Benoist \cite{Benoist94} and \cite{BenTor} made significant progress on real projective and affine structures on nilmanifolds. Also, Dupont \cite{Dupont} classified real projective structures on the class of $3$-manifold that are quotients of Sol by lattices. 

We believe that there will be further progress of some significance in this field: 
One interesting question of Benoist \cite{BenTor} is whether for any closed manifold $M$, there exists an $n$-dimensional torus $T^n$ so that $M \times T^n$ admits an affine or real projective structure with a linear diagonalizable holonomy group. 

Recently, Dennis Sullivan has produced some types of projective structure on 
$\rpn \# \rpn$ which he talked about informally at a Stony Brook class in 2024. However, so far he has not claimed that this is a real projective structure in our sense given here using charts.
These might be modeled on some other geometries.  



\subsection{Comments on the results of Benoist} 

To better understand the results of this paper, one probably needs to understand the paper \cite{BenTor} and the more technical preliminary paper \cite{Benoist94} of Benoist.  

We give a more detailed summary of \cite{Benoist94} in Sections \ref{subsec:auto}. Roughly speaking, 
his main tool here is that one can always lift any projective action of a nilpotent or abelian group to the local action of an identity component of 
the nilpotent Zariski closure $I$. 
This is essentially because the actions are proper on its orbits.
Hence, the orbit of a covering group of $I$ maps to an $I$-orbit in $\SI^n$ as a covering map
mainly because the action is proper on an orbit. 
This extends the old idea of Nagano-Yagi \cite{NY}, which are just the abelian cases
of this phenomenon.

Now, we try to summarize \cite{BenTor}. 
Section 1 of \cite{BenTor} reviews the affine and real projective structures particularly for $2$-tori. 
A good summary of the second paper \cite{Benoist94} needed here is in Section 2 of \cite{BenTor}. In fact, we will only need positive digonalizable linear Lie groups rather than the full theory. This is also well explained in Sections 3 and 5 of \cite{BenTor}. 
The main aim of Benoist was to extend the $2$-dimensional results to higher-dimensional manifolds 
with nilpotent holonomy groups. 
However, Section 4 of \cite{BenTor} contains some claims that need 
further justifications on which some people currently are working. These claims of Benoist have not been exploited anywhere as of now.


\subsection{Outline} 


In Section \ref{sec:geostr}, we will review some facts on geometric structures, 
and the octantization of Benoist: 
For the first half of this paper, we are trying to reprove Theorem 
4.9 of \cite{BenTor} classifying the real projective manifolds with infinite-cyclic holonomy groups.
This is Proposition \ref{prop:mainp}. 
We will approach this in two steps by first proving Theorem \ref{thm:main} for 
regular diagonalizable holonomy manifolds and then we prove the corollary using 
perturbation.

In Section \ref{subsec:auto}, 
we summarize some nilpotent Lie group action materials from \cite{Benoist94}. 
In Section \ref{subsec:Benoist}, we discuss the constructions of real projective manifolds 
by Benoist \cite{BenTor}.  
In Section \ref{subsec:split}, we discuss the splitting of real projective manifolds along
\cite{psconv}. Then we will define elementary Benoist manifolds and Benoist manifolds, 
as discovered first by Benoist. 
In Section \ref{subsec:top}, we will go into the topology of Hopf and Benoist manifolds.
In Section \ref{subsec:Prop4.9}, we restate Proposition 4.9 of \cite{BenTor}, which will be fully proved in Section \ref{subsec:general}. 

In Section \ref{sec:delta}, we will discuss manifolds
whose holonomy group is infinitely cyclic and regular diagonalizable. 
In Section \ref{subsec:simplices}, we discuss the possibility of 
the union of a simplex and its sides so that the resulting 
quotient space is a compact manifold with corner and lifts to the universal cover. 
In Section \ref{subsec:ambient}, we show that such a union always implies the existence of some canonical ambient manifolds in the universal cover. 
In Section \ref{subsec:canonical}, 
we classify closed real projective manifolds with infinite-cyclic regular diagonalizable holonomy groups. This proves Theorem \ref{thm:main}. 
In Section \ref{subsec:general}, we will present a proof for Proposition \ref{prop:mainp}. This proof is based on the concept introduced by Cooper and Goldman, involving the perturbation of holonomy groups to diagonalizable groups.

For the rest of the paper, we are trying to prove our nonexistence results. The main tool is the Kuiper completion for geometric structure using the theories in our old work \cite{psconv}.
In Section \ref{sec:obstruct}, we establish the proof for Corollary \ref{cor:main}.
In Section \ref{subsec:generalcov}, we will present the general case, utilizing the Kuiper completion of the holonomy cover.
In Section \ref{subsec:counter}, we will discuss some illustrative examples.

We thank Yves Benoist, Daryl Cooper, William Goldman, Arielle Leitner, and Max Riestenberg for helpful comments regarding the results of this paper.


\section{Geometric structures and Benoist manifolds} \label{sec:geostr}

\subsection{Geometric structures and the deformation spaces}\label{subsec:geostr} 

We let $\Pi: \bR^{n+1} - \{O\} \ra \SI^n$ denote the natural projection. 
$\SL_pm(n+1, \bR)$ acts on $\SI^n$ by an induced action on $\bR^{n+1}$. 
Let $X = \SI^n$ and $G = \SL_\pm(n+1, \bR)$ act on $X$.  
We consider the pair $(\SI^n, \SL_\pm(n+1, \bR))$ as the real projective geometry. 

An {\em affine patch} in $\SI^n$ is an open $n$-hemisphere identifiable with an 
affine subspace of $\rpn$ under the double cover. 
A {\em simplex} in $\SI^n$ is the convex hull of a finite set of points in a general position 
in an affine patch of $\SI^n$. 
An {\em open simplex} is the manifold-interior of a simplex considered as a manifold with corner. 
Let $\Delta$ denote a group of positive diagonal matrices with respect to a coordinate system. Each open orbit of $\Delta$ is $X$ is called an {\em $n$-octant}, which is an $n$-simplex.
It is easy to see that each open side of an $n$-octant is an orbit of $\Delta$ as well. 

We follow Section 1 of \cite{Benoist94}.  
A real projective structure on a manifold $M$ is given by a maximal atlas of charts to $\SI^n$ with transition maps in $ \SL_\pm(n+1, \bR)$. 
Given a manifold $M$, we are going to denote by $\tilde M$ its universal cover. 
With a real projective structure, there is an associated pair of 
an immersion $\dev: \tilde M \ra \SI^n$ and
a homomorphism $h: \pi_1(M) \ra G$
where $\dev \circ g = h(g) \circ \dev$ for each $g \in \pi_1(M)$. 
The pair $(\dev, h(\cdot))$ is determined up to the action 
\[g \cdot (\dev, h(\cdot)) = (g\circ \dev, g \circ h(\cdot) \circ g^{-1}) \hbox{ for }
g \in G.\] 

In our paper, we will use the {\em holonomy cover} $M_h:= \tilde M/K_h$ where 
$K_h$ is the kernel of $h$. 
Then $\Gamma_h:= \pi_1(M)/K_h$ acts as a deck transformation group of $M_h \ra M$. 
There is a pair of an immersion and a homomorphism 
\[(\dev_h: M_h \ra \SI^n, h_h: \pi_1(M)/K_h \ra \SL_\pm(n+1, \bR))\]
satisfying $\dev_h \circ \gamma = h_h(\gamma)\circ \dev_h$ for 
each $\gamma \in \Gamma_h$. 


Let $M_{K'}:= \tilde M/K'$ where $K_h\supset K'$. 
Covers such as $M_{K'}$ are said to be the {\em covers of holonomy covers}. 
Let $N(K')$ denote the smallest subgroup containing $K'$ in $\pi_1(M)$. 
Then there is an immersion $\dev_{K'}:M' \ra \SI^n$ induced from $\dev$ and 
$h_{K'}:\pi_1(M)/N(K') \ra \SL_\pm(n+1, \bR)$ induced from $h$. 
Let $\Gamma_{K'}$ denote a deck transformation group.
acting on $M_{K'}\ra M$. Note that $\Gamma_{K'}$ is isomorphic to $N(K')/K'$
where $N(K')$ is the normalizer of $K'$. 



For a closed manifold $N$, 
a {\em primitive deformation space} $\Def''(N)$ of 
{\em real projective structures} is defined as 
the space of $(D, h)$ where 
\begin{itemize} 
\item $D: \tilde N \ra \SI^n$ is an immersion, and 
\item $h: \pi_1(N) \ra \SL_\pm(n+1, \bR)$ satisfying 
$D\circ \gamma = h(\gamma) \circ D$ for every $\gamma \in \pi_1(N)$.
\end{itemize} 
The topology is given by the $C^r$-topology for $r \geq 2$.  
An isotopy $\iota: \tilde N \ra \tilde N$ is {\em equivariant} if 
$\gamma \circ \iota = \iota \circ \gamma$ for any $\gamma \in \pi_1(N)$. 
The group $I_{\pi_1(N)}(\tilde N)$ of equivariant isotopies on $\tilde N$ acts on 
$\Def''(N)$ by the right action
\[(D, h) \mapsto (D\circ \iota, h) \hbox{ for } \iota \in I_{\pi_1(N)}(\tilde N).\] 
We give the quotient topology to the quotient space $\Def'(N)$ of $\Def''(N)$.

The Ehresmann-Thurston principle states that 
\begin{equation}\label{eqn:ET} 
\Def'(M) \ra \Hom(\pi_1(M'), \SL_\pm(n+1, \bR))
\end{equation} 
is a local homeomorphism for a closed $n$-manifold $M$.
(See \cite{Goldman88}, \cite{dgorb}, and \cite{CLM18}.) 
The quotient of this space by the conjugation action of $\SL_\pm(n+1, \bR)$ is the 
{\em deformation space} $\Def(N)$ of $N$. 
(Note that the conjugation action is the left action and the isotopies are the right action, and so they commute.)

The {\em open join} of two disjoint sets $A$ and $B$ on $\SI^n$ 
is given as 
\[ 
A \ast_o B := \{ [t \vec{x} + (1-t) \vec{y}]| [\vec{x}] \in A, [\vec{y}]\in B, t \in (0, 1)\}.
\]
The {\em open join} of two abstract disjoint sets $A$ and $B$
is given as 
\[ 
A \ast_o B := \{ (x, y, t)| x\in A, y\in B, t\in (0, 1)\}.
\]
Of course, this is diffeomorphic to $A\times B \times (0, 1)$ as a smooth manifold
when $A$ and $B$ are smooth manifolds.

We will need: 
\begin{lem}\label{lem:p=1} 
Let $M$ be a generalized Hopf manifold of type $(p, q)$ 
with $1 \leq p \leq q \leq n-2$ and $p+q = n-1$ 
with infinite-cyclic $h(\pi_1(M))$. 
Then we obtain that 
$\pi_1(M) = \langle \gamma, K\rangle$ for 
an element $\gamma \in \pi_1(M)$ and the kernel $K$ of $h$, 
and the following hold\/{\em :} 
\begin{itemize} 
\item the induced map
$\dev_h: M_h \ra \SI^n \setminus S^p \setminus S^q$ 
from the developing map $\dev$ is a finite covering map
for two mutually disjoint great spheres $S^p$ of dimension $p$ and $S^q$ of dimension $q$,
and $M_h$ is diffeomorphic to $\hat S^p \times \hat S^q \times \bR$ for a finite 
covers $\hat S^p$ of $S^p$ and $\hat S^q$ of $S^q$. 
\item if $p \geq 2$, then $\dev_h$ is a diffeomorphism and $M_h = \tilde M$,
and $\tilde M$ is diffeomorphic to $S^p \times S^q \times \bR$. 
\end{itemize} 
\end{lem} 
\begin{proof} 
Suppose $p=1, q=n-2$. 
Here, $\pi_1(M) = \langle \gamma, K \rangle$. 
Since 
$\SI^n \setminus S^1 \setminus S^{n-2}$ equals 
$S^1\ast_o S^{n-2}$, 
$\tilde M$ can be identified as a join 
$\tilde S^1 \ast_o \tilde S^{n-2}$
for universal covers $\tilde S^1$ and $\tilde S^{n-2}$. 
Since $K$ acts trivially on $S^1$ and $S^{n-2}$ having the trivial holonomy image,
it acts as deck transformation groups on $\tilde S^1$ and $\tilde S^{n-2}$ respectively.
Now, $K$ acts on $\tilde S^1\times \tilde S^{n-2}$
cocompactly since otherwise we have only automorphisms 
induced by $\gamma^n, n\in \bZ,$
and we cannot obtain the cocompactness of the action on 
the entire space. 
Hence, the conclusion follows.  

If $p\geq 2$, then $2 \leq p \leq q \leq n-3$ and $S^p$ and $S^q$ are both simply connected, 
and $M_h =\tilde M$. Since $\dev_h$ is a covering map to $\SI^n \setminus S^p \setminus S^q$, it has to be a diffeomorphism. 
\end{proof} 

\subsection{Projective automorphism groups} \label{subsec:auto} 
We summarize some materials needed from \cite{Benoist94}. 


\begin{defn} \label{defn:Is} 
Let $W$ be a real projective manifold with a projective immersion $i_W:W \ra \SI^n$. 
An automorphism $\phi: W \ra W$ is {\em projective} if 
there is an element $\phi'\in  \SL_\pm(n+1, \bR)$ so that 
$i_W \circ \phi = \phi' \circ i_W$. 
\begin{itemize} 
\item We denote by $Is(W)$ the group of projective automorphisms of $W$
and by $Is_0(W)$ the identity component of $Is(W)$.
\item There is a homomorphism $h_{i_W}: Is(W) \ra \SL_\pm(n+1, \bR)$
defined by $h_{i_W}(\phi) \circ i_W = i_W \circ \phi$ for $\phi \in Is(W)$. 
\item We denote by $K_{i_W}$ the subgroup of elements $\phi$ 
of $Is(W)$ where $h_{i_W}(\phi) = \Idd$. 
\item We define  $K_{i_W, 0} = K_{i_W} \cap Is_0(W)$,  
$Is(i_W)  := h_{i_W}(Is(W)) \cong Is(W)/K_{i_W}$, and  
$Is_0(i_W):= h_{i_W}(Is_0(W)) \cong Is_0(W)/K_{i_W, 0}$,
the identity component of $Is(i_W)$. 
\item For a subgroup $L$ of $Is(i_W)$, 
we denote by $Is_L(W)$ 
the subgroup of $Is(W)$ whose image under $h_{i_W}$ is in $L$. 
\end{itemize} 
\end{defn} 
Note $K_{i_W}$ is discrete since it goes to the trivial element of $\SL_\pm(n+1,\bR)$. 

For a real projective manifold $M$, 
we give as usual the analytic structure on $\tilde M$ induced by 
$\dev$ and the analytic structure descends to $M$  since the holonomy group action 
is analytic.

We say that a group $G$ acting on $\SI^n$ {\em lifts} 
to $M':= M/K'$ where $K' \subset K_h$.  
if there is a covering group $\rho:G' \ra G$ with the kernel $K_{G'}$ 
that acts on $M'$ so that for each element $g\in G'$, 
$\dev \circ g = g'\circ \dev$ for element $g'\in G$. 
Here, we say that $G'$ is a {\em lift} group. 

An orbit $\mathcal{O}$ of $G$ {\em lifts} if a covering group $\tilde G$
lifts with an orbit $\tilde{\mathcal{O}}$ so that 
$\dev|\tilde{\mathcal{O}} \ra \mathcal{O}$ is a covering map.

Let $M$ be a closed $\rpn$-manifold with the nilpotent holonomy group $H$.
We can choose a maximal connected nilpotent subgroup $N$ of $G$ so that 
$[H: H\cap N]$ is finite.  We can do this by taking the identity component of the Zariski closure of $H$ and taking an ambient maximal connected nilpotent subgroup.

\begin{lem}[Lemma 2.1 of \cite{Benoist94}]\label{lem:lifting} 
Suppose that $M$ is a closed $\rpn$-manifold with a nilpotent holonomy group. 
Then for a choice of a maximal connected nilpotent Lie group $N$ of $G$
so that \[[h(\pi_1(M)): h(\pi_1(M)) \cap N] < \infty,\]  
we have $N \subset Is_0(\dev), Is_0(\dev_h)$ 
\end{lem} 
We added $Is_0(\dev_h)$ because there is no difference in the proofs and will continue to do so afterward. 

We say that $M$ is {\em $n$-octantizable} if $\Delta \subset Is(\dev)$ for a diagonal matrix group 
for a choice of frame of $\bR^{n+1}$.  
Since $\Delta$ is simply connected, this condition is equivalent to 
$\Delta \subset Is(\dev_{K'})$ for any $K'$ where $K_h \supset K'$. 
It is easy to see that $Is(\dev_h)$ contains a copy of $\Delta$
as a subgroup, and hence $M_h$ is a union of mutually disjoint orbits of $\Delta$.
In addition, $\Delta \subset Is(\dev_{K'})$ and $M_{K'}$ is a union of mutually disjoint orbits of
$\Delta$. 

We say that a diagonalizable projective automorphism with $n+1$ mutually distinct eigenvalue norms is {\em regular diagonalizable}. Any group generated by such an element is said to be {\em regular diagonalizable}.

\begin{prop}\label{prop:diagc} 
Suppose that $h(\pi_1(M))$ is generated by a regular positive diagonalizable element. Then for the unique full positive diagonalizable group $\Delta$ containing 
$h(\pi_1(M))$, we have $Is_0(\dev_{K'}) = \Delta$, and the $\Delta$-action lifts to $M_{K'}$ for $K_h \supset K'$ 
and $\tilde M$ is a union of $\Delta$-orbits that 
maps diffeomorphic to open simplices under $\dev_{K'}$. 
\end{prop} 
\begin{proof} 
It is clear that $\Delta$ is the unique maximal connected nilpotent subgroup. 
Hence, $\Delta \subset Is_0(\dev)$ and $\subset Is_0(\dev_{K'})$. 
Since $\Delta$ is simply connected, the $\Delta$-action lifts to $\tilde M$ and, hence, to 
$M_{K'}$. Also, Since $\Delta$ acts freely on a simplex, 
each $\Delta$-orbit maps to a simplex in $\SI^n$. 
\end{proof} 

\begin{prop}[from Proposition 1 of Benoist \cite{Benoist94}]  \label{prop:oct1} 
A closed $\rpn$-manifold $M$ with a diagonalizable holonomy
is octantizable. \qed
\end{prop} 


We denote by $\clo_X(A)$ the closure of a subset $A$ in the space $X$. 
Heuristically, the following Propositions \ref{prop:oct2} and \ref{prop:Lem3_1}. 
gives us a decomposition of a finite cover of a
closed octantizable $\rpn$-manifold into ``bricks''.

\begin{prop}[Proposition 2 of Benoist \cite{Benoist94}] \label{prop:oct2} 
Let $M$ be an octantizable closed $\rpn$-manifold. 
Let $M_h$ be a holonomy cover. 
Let $\hat \Omega$ denote an open orbit of $Is_0(M_h)$ in $M_h$. 
\begin{itemize} 
\item The restriction $\dev|  \clo_{M_h}(\hat \Omega)$ to its image is a covering 
with the deck transformation group $K_{\dev, 0}$. 
\item Let $\Gamma_{K, 0} = \Gamma \cap \langle K, Is_0(M_h)\rangle$, and let 
$\Gamma_{0}:= \Gamma \cap Is_0(M_h)$. 
Then $\clo_{M_h}(\hat \Omega)/\Gamma_{0}$ is a manifold with corner, and 
embeds in a compact submanifold with corner in $M_h/\Gamma_{K, 0}$ under 
the projection $M_h \ra M_h/\Gamma_{K,0}$.
\end{itemize} 
\end{prop}

\begin{prop}[Lemma 3.1 of \cite{Benoist94}] \label{prop:Lem3_1}
	Let $M$ be a closed $\rpn$-manifold with a holonomy cover $M_h$
	with the deck transformation group $\Gamma$. 
	\begin{itemize} 
		\item 
$K$ and $Is_{Is_0(\dev) }(M_h)$ commute.
$Is_{Is_0(\dev)}(M_h)$ is identical with 
		$K \times_{K_{\dev,0}}  Is_0(M_h)  =  \langle K, Is_0(M_h)\rangle$.
		\item If $M$ is octantizable, then 
		$Is_{Is_0(\dev)}(M_h)$ is a finite-index subgroup $Is(M_h)$. 
		\item $\Gamma_{K, 0}= \Gamma \cap Is_{Is_0(\dev) }(M_h)$ is a finite-index subgroup of 
		$\Gamma$. 
		\end{itemize} 
	\item 
	\end{prop} 
We denote by $\Gamma_{K, 0}:= \Gamma  \cap Is_{Is_{\dev, 0 } }(M_h)$.  
We call $\clo_{M_h}(\hat \Omega)/\Gamma_{K, 0}$ a {\em brick}.

\subsection{Benoist manifolds or $A_g$ or $B_g$} \label{subsec:Benoist} 

For an expository purpose of this paper, we need to introduce some notations and names. 
These are by no means standard yet as there are such a few number of people working in 
this area. 


\begin{defn}[$A_g$ and $B_g$ in Example 4.8 of  Benoist \cite{BenTor}]
Let $g$ be an element of $\SL(n+1, \bR)$ with the norms of eigenvalues 
$\lambda_1, \dots, \lambda_{n+1}$ in an nonincreasing order where 
we repeat these according to the multiplicities. 

Suppose that $\lambda_i$ is an eigenvalue of $g$
with $E_0:= \ker(g - \lambda_i)$ with dimension $1$. 
For each eigenvalue $\alpha$ of $g$, we denote by 
$r_\alpha$ the its multiplicity in the characteristic polynomial of $g$. 
We define $E_+$ to be the direct sum of cyclic spaces $\ker (g - \alpha \Idd^r_\alpha)$ 
with $|\alpha| > \lambda_i$, 
and $E_-$ to be the direct sum of cyclic spaces 
$\ker (g - \alpha \Idd)^r_\alpha$ with $|\alpha| < \lambda_i$. 

Let us orient $E_0$ arbitrarily. 
We denote by $H$ the hemishere bounded by $\Pi(E_+ \oplus E_-)$ and 
the point of $\SI^n$ in the positive direction of $E_0$. 
We define $\Sigma_{\pm} := \Pi(E_\pm)\cap H$, and 
$\Sigma_{\pm, 0} := \Pi(E_\pm \oplus E_0)$.

We define \[U_\pm := H \setminus (\Sigma_{\pm} \cup \Sigma_{\mp, 0}).\]
We let $d_\pm: = \dim E^\pm$.
$U_+$ is diffeomorphic to $(0, 1) \times \SI^{d_+-1} \times B^{d_-}$, and 
$U_-$ to $(0, 1) \times B^{d_+} \times \SI^{d_--1}$, 
$\partial U_+ = \partial U_-$ is diffeomorphic to $(0, 1) \times \SI^{d_+ -1} \times \SI^{d_--1}$. 

For $d_+, d_-\geq 1$, $U_\pm$ is connected. 
$A_g := U_+/\langle g \rangle$ and $B_g:= U_-/\langle g \rangle$. 
if $d_+ =1$ or $d_- =1$, $U_\pm$ is disconnected.
We take a component $V_\pm$ and define $A_g:= V_+/\langle g \rangle$ and 
$B_g:= V_-/\langle g \rangle$. 
\end{defn}

A basic problem with this notation in this paper is that it does depend on $g$, and we have 
to change $g$ often.
We introduce the definitions for the purpose of this paper:

\begin{defn}\label{defn:eBenI}  
An {\em $n$-bihedron} $Q$ is the closure of a component of $\SI^n \setminus S \setminus S'$ for two transversal great $(n-1)$-spheres $S$ and $S'$. (See Chapter 3 of \cite{psconv}.) The union of two closed $(n-1)$-dimensional hemispheres $H'$ and $H''$ forms the boundary of $Q$. We refer to these as the {\em boundary $(n-1)$-hemispheres}. 
$S$ contains a boundary $(n-1)$-hemisphere, and let $q$ be a point in the interior of the other boundary $(n-1)$-hemisphere. 

{\em Elementary Benoist manifolds of type I} are compact $\mathbb{RP}^n$-manifolds covered by $Q \setminus S \setminus \{q\}$ for some choice of $Q, S, q$ as described above, where the deck transformation group is generated by projective transformations acting on $S$ and ${q, q_-}$.

The real projective manifolds that are real projectively diffeomorphic
to $Q \setminus S \setminus \{q\}$ are called {\em elementary pre-Benoist manifolds of type I}. 
\end{defn} 

\begin{defn} \label{defn:eBenII} 
Let $H^n$ be a closed n-hemisphere with its boundary containing a great sphere $S^{n-2}$ of dimension $n-2$. In addition, its interior intersects a great circle $S^1$ that is disjoint from $S^{n-2}$.

An {\em elementary manifold of Benoist type II} refers to a compact real projective manifold that is 
covered by a finite cover of $H^{n, o} \setminus S^{n-2} \setminus S^1$ such that the deck transformation group is generated by projective transformations acting on $S^{n-2}$, $S^1$, and $H^n$. (For $n\geq 4$, we do not need to consider finite covers.)

The real projective manifolds that are real projectively diffeomorphic
to a finite cover of $H^{n, o} \setminus S^{n-2} \setminus S^1$ 
are called {\em elementary pre-Benoist manifolds of type II}. 
\end{defn}

\begin{defn} \label{defn:eBenIII} 
Let $H^n$ be a closed $n$-hemisphere with its boundary containing a great sphere $S^{q}$ of dimension $q$, $1 \leq q \leq n-3$. 
Additionally, its interior $H^{n, o}$ intersects a great circle $S^p$ that is disjoint from $S^q$ where $p+q = n-1, 2 \leq p \leq n-2$. 

An {\em elementary manifold of Benoist type III} refers to a compact real projective manifold that is 
covered by a finite cover of $H^n \setminus S^p \setminus S^q$ such that the deck transformation group is generated by projective transformations acting on $S^p$, $S^q$, and $H^n$. 
Here, $(p, q)$ are said to be the characteristic of this elementary manifold of Benoist type III. (For $p \leq n-3$, we do not need to consider finite covers.)

The real projective manifolds projectively diffeomorphic 
to a finite cover of $H^{n, o} \setminus S^p \setminus S^q$ 
are called {\em elementary pre-Benoist manifolds of type III}. 
\end{defn} 

Notice that elementary pre-Benoist manifolds of type I and III have connected boundary, while elementary pre-Benoist manifold of type II has two boundary components. 
Each of these of type I and II 
has a boundary component that is a pure Hopf manifold of dimension $n-1$.
The type III one has a generalized Hopf manifold of dimension $n-1$ of type $(p-1, q)$
 as the boundary. 

For a regular positive diagonalizable 
projective automorphism $g$, Example 4.8 of \cite{BenTor} give
$A_g$ with $d_+=1$ and $B_g$ with $d_-=1$
 as elementary Benoist manifolds of type I. 
Other types termed $A_g$ or $B_g$ are all elementary Benoist manifods of type II or III. 


An example of an elementary Benoist manifold of type I can be constructed by taking a quotient of $Q \setminus S \setminus \{q\}$ by $\langle g \rangle$, where $g$ is a projective transformation acting on $S$ and $S'$, and $q$ is associated with the eigenvalue of largest or smallest norm with multiplicity $1$.





An example of an elementary Benoist manifold of type II can be constructed by taking a quotient of $H^{n, o} \setminus S^{n-2} \setminus S^1$ by $\langle g \rangle$, where $g$ is a projective transformation acting on $H^{n, o}$, $S^{n-2}$, and $S^1$. Here, $S^1$ and $S^{n-2}$ correspond to the direct sum components of Jordan decompositions, and the norms of the associated eigenvalues of one component are strictly larger than those of the other components.

\subsection{Splitting of real projective manifolds}  \label{subsec:split}
We say that a real projective manifold $M$ {\em splits} along
a totally geodesic locally two-sided $\mathbb{RP}^{n-1}$-manifold $N$ of codimension-one
to a real projective manifold $M'$ 
if $M \setminus N$ completes to $M'$ using the path metric on
$M \setminus N$
induced from the Riemannian metric of $M$.  (See \cite{psconv}.)
We say that $M$ {\em splits} into
components of $M'$. 
We say that $M$ {\em decomposes into $M'$ along $N$} if 
each component of $M'$ is real projectively isomorphic to 
closures of components of $M-N$. 
We may say that $M$ {\em decomposes} into components of $M'$ also.

Let a finite sequence of compact real projective manifolds $M_1, \dots, M_k$
be obtained as a decomposition of $M$ by splitting along a totally geodesic
hypersurface $S$. 
Suppose that each $M_i$ has at most two boundary components. 
Such is the case for pre-Benoist manfiolds and elementary Benoist manifolds. 
We can 
have only the following possibilities: 
\begin{itemize} 
\item We say that $M_1, \dots, M_k$
are {\em arranged} in a {\em  sequence} if 
each boundary of $M_1$ has exactly one or two boundary components and 
for each component of $S$, $M_i^o$ and $M_{i+1}^o$ are adjacent and 
$i$ is unique for that component. 
If the indices are given mod $k$, then we say that the sequence is {\em circular}.
Otherwise, it is {\em linear}. 
\item We say that $M_1, \dots, M_k$
are {\em arranged} in a {\em circular} (resp. {\em linear}) 
{\em sequence up to double coverings} 
if there is a one-fold or two-fold cover $M'$ of $M$ and 
a circular  (resp. linear) decomposition along the inverse image $S'$ of $S$.
\end{itemize}


\begin{defn} \label{defn:Benoist}
A closed real projective manifold that splits into finitely many
elementary Benoist manifolds is said to be a {\em Benoist manifold}. 
A {\em pure Benoist manifold} is a Benoist manifold that is not a generalized Hopf manifold
or a pure Hopf manifold. 
\end{defn}

An example of a Benoist manifold can be obtained by taking a finite sequence of type II ones with 
projective automorphism $g$ where
we need to switch between the hemispheres of dimension $n$ sharing the boundary components
and  finally  obtain the closed manifold by coming back to the beginning one by 
attaching projectively. 

Another example can be made by starting as above, but 
we can use two type I ones to  the end of a finite sequence. 
(These constructions of Benoist generalized the construction of Goldman \cite{Gthesis}.) 

We summarize some obvious properties of Benoist manifolds. 

\begin{prop}\label{prop:propBen}  $ $ 
Assume $n \geq 3$. 
Let $M$ be a Benoist manifold. 
Then the following holds\/{\em :} 
\begin{itemize} 
\item If a real projective manifold finitely covers $M$, it is Benoist also. 
\item $M$ is covered by a real projective manifold that decomposes into elementary
Benoist manifolds.  
\item  One of the following holds\/{\em :} 
\begin{itemize} 
\item $M$ is a pure Hopf or a generalized Hopf manifold. Furthermore, 
$\dev_h$ maps $M_h$ diffeomorphic to $H^{n, o} \setminus \{x\}$ for a hemisphere $H^n$ and an interior point $x$, or 
$\dev_h$ is a diffeomorphism or a finite covering map to $\SI^n\setminus S^p \setminus S^q$,  $p+q = n-1$, $2 \leq  p \leq n-2, 1 \leq q \leq n-3$ for disjoint great spheres $S^p$ and $S^q$. 
\item A finite cover of $M$ decomposes into elementary Benoist manifolds
which are as follows\/{\rm :} 
\begin{itemize} 
\item two type III ones 
meeting one another at their common boundary. 
Furthermore, $W_h$ is projectively diffeomorphic to an open submanifold of $\SI^n$.
Here, the characteristics $(p,q)$ and $(p', q')$ of the adjacent ones
satisfy one of
$p'=p, q'=q$ or $p'= q+1, q'= p-1$, 
where $2 \leq p, p' \leq n-2, 1 \leq q, q' \leq n-3$.
\item type II ones in circular sequence, which is a generalized Hopf manifold
where $\dev_h$ is 
a finite covering map of $\SI^n -S^1 - S^{n-2}$
for disjoint great spheres $S^1$ and $S^{n-2}$, 
or 
\item type II ones but ending with two type I ones   
\end{itemize} 
\end{itemize} 
\end{itemize} 
\end{prop} 
\begin{proof} 
Let $N$ be the submanifold so that $M \setminus N$ completes to a union of 
elementary Benoist manifolds. Let $\tilde M$ denote the universal cover and
let $\tilde N$ denote the inverse image of $N$. 
Then the first item is clear. 

Let $M'$ be a finite cover of $M$. Let $N'$ be the inverse image.
Then $M'\setminus N'$ covers the interior of elementary Benoist manifolds. 
Clearly, the completions cover the elementary Benoist manifolds. 

For the second item, 
each elementary Benoist manifold of type II can be covered by one with two boundary components, which is a $2$-fold cover. 
We do this for each elementary Benoist manifold of type II in
the splitting. Then we glue back to obtain a finite cover of $M$.

The first item of the last item and the first starred item follow by Lemma \ref{lem:typeIII}.

The last two starred items follow by considering how an infinite or circular sequence of
pre-Benoist manifolds of type II develops to $\SI^n$. 

\end{proof}

\begin{lem}\label{lem:typeIII}
Suppose that $M$ splits so that an elementary Benoist manifold of type III of type $(p, q)$ occurs. 
Then there are only one or two elementary Benoist manifolds in the split manifolds
\begin{itemize}
\item It could be a generalized Hopf manifold covered by
a finite cover of 
$\SI^n \setminus \SI^p \setminus \SI^q$ for $2\leq p \leq n-2, 1\leq q\leq n-3$, and $p+q=n-1$, or 
\item there is an adjacent elementary Benoist manifolds of type III 
with characteristic $(p', q')$, $p'= q+1, q'= p-1$,  
where $2 \leq p, p' \leq n-2, 1 \leq q, q' \leq n-3$.
\end{itemize}
Furthermore,  the holonomy cover $M_h$ of $M$ 
is projectively diffeomorphic to its image on $\SI^n$ under $\dev_h$ in the second case. 
\end{lem}
\begin{proof} 
First, suppose that there is an elementary pre-Benoist manifold of 
type III one with characteristic $(p, q)$ in $M_h$.
It has a connected boundary homotopy equivalent to 
$\SI^{p-1} \times \SI^q \times \bR$ for $2 \leq p \leq n-2, 1\leq q \leq n-3$. 
Since only type III ones have this topological type boundary property, 
there is a unique adjacent elementary pre-Benoist manifold,
must be of type III, and they must share their boundary. 

Therefore, these respectively map as finite convering maps 
\begin{itemize} 
\item to $H^n_+ \setminus B^p \setminus \SI^q$ 
where $H^n_+$ is a hemisphere bounded by a great sphere $S^{n-1}$, and 
$S^q$ is a great sphere in $\partial H^n_+$ for $B^p = S^p \cap H^n_+$
for a great sphere $S^p$ disjoint from $S^q$, and 
\item 
to $H^n_- \setminus B^{p'} \setminus S^{q'}$ 
$H^n_-$ is the other hemisphere bounded by a great sphere $S^{n-1}$, and 
where $S^{q'}$ is a great sphere in $\partial H^n_-$ 
for $B^{p'} = S^{p'} \cap H^n_-$
for a great sphere $S^{p'}$ disjoint from $S^{q'}$. 
\end{itemize} 
Since the boundary of the two elementary pre-Benoist manifolds must agree,  
there are only two possibilites: 
\begin{itemize}
\item $S^{n-1} \cap B^p = S^{n-1} \cap B^{p'}$ and 
$S^q = S^{q'}$.
\item $S^{n-1} \cap B^p = S^{q'}$ and $S^{n-1} \cap B^{p'} = S^q$. 
\end{itemize} 

In the first case, since $g$ acts on $S^p, S^q, S^{p'}, S^{q'}, S^{n-1}$, we must have 
$S^p = S^{p'}$ and $S^q = S^{q'}$. 
Then $M_h$ maps under $\dev_h$ as a finite covering to
$\SI^n - S^p -S^q$ for $2 \leq p \leq n-2, 1 \leq q \leq n-3$
where $(p, q)$ is the characteristic and and $p+q = n-1$. 

In the second case, $(p', q')$ should be $p'= q+1$ and $q'= p-1$
so that the boundary agrees. 
In this case, since the covering maps has to agree as a covering map
on the boundary of the type III ones, the covering maps are diffeomorphisms.  
It follows that $W_h$ is projectively diffeomorphic to its image 
in $\SI^n$. 
\end{proof}

\begin{cor} \label{cor:propBen} 
A Benoist manifold is finitely covered by one of the following\/{\rm :} 
\begin{itemize} 
\item[(i)] a union of two type III elementary 
Benoist manifolds, 
\item[(ii)]  a union of type II elementary Benoist manifolds in a finite circular sequence which is a generalized Hopf manifold. 
\item[(iii)] a union of type I or II elementary Benoist manifolds in a finite linear sequence
where there are two type I ones at the end. 
\end{itemize} 
\end{cor} 
\begin{proof}
This is immediate from Proposition \ref{prop:propBen}. 
\end{proof} 

Note that only (i) and (iii) correspond to pure Benoist manifolds.

\subsection{Topology of Hopf and Benoist manifolds}\label{subsec:top} 

Recall the definition of pure Hopf and generalized Hopf manifolds, and 
those of Benoist manifolds from Section \ref{subsec:Benoist}.
We  need some facts: 

\begin{prop}\label{prop:topology} 
Assume $n \geq 3$. 
\begin{itemize} 
\item 	A pure Hopf manifold is diffeomorphic to a manifold finitely covered by 
 $\SI^1 \times \SI^{n-1}$. 
\item 	A generalized Hopf manifold of type $(p, q)$ is diffeomorphic to a manifold
finitely covered by  $\SI^p \times \SI^q \times \SI^1$, $1\leq p \leq q\leq n-2, p+q=n-1$. 
\item 	A Benoist manifold is diffeomorphic to 
\begin{itemize} 
\item a manifold finitely covered by
 $\SI^1\times \SI^1 \times \SI^{n-2}$, which is also a generalized Hopf manifold, 
	or
\item a manifold finitely covered by 
$\SI^1 \times \SI^{n-1}$ 
\end{itemize} 
respectively depending on whether 
	the decomposition contains no elementary Benoist manifold of type I or not. 
\end{itemize} 
\end{prop} 
\begin{proof} 
%


Let us demonstrate that a finite cover of the holonomy cover of each of these manifolds is diffeomorphic to $N \times \mathbb{R}$, where $N$ is a closed manifold.

By using Lemma \ref{lem:p=1}, we can conclude that for pure Hopf or generalized Hopf manifolds, there exist covering spaces that are diffeomorphic to $\mathbb{R} \times \mathbb{S}^{n-1}$ or $\mathbb{R} \times \mathbb{S}^p \times \mathbb{S}^q$, respectively, where the inclusion of $\mathbb{R}$ arises from the join construction.

We go to the other possibilities for Benoist manifolds according the Corollary \ref{cor:propBen}: 
We take a finite cover decomposing into elementary Benoist manifolds. 

Suppose that there is an elementary Benoist manifold of type III. Then we have a union of
two elementary pre-Benoist manifolds of type III as the holonomy 
cover. 
By the proof of Lemma \ref{lem:typeIII}, 
we easily see that the union is projectively diffeomorphic to 
$\SI^n \setminus B^i \setminus B^j$ 
for two balls $B^i$ and $B^j$ of dimension $i$ and $j$ respectively for $i+j = n$. 
Then this open region is diffeomorphic to $\SI^{n-1}\times \bR$ as the balls are contractible by deformation retractions.

Suppose now that at the end of a linear sequence, up to double covering, there are two elementary Benoist manifolds of type I. In this case, the overall manifold is a union of a linear sequence of manifolds that are diffeomorphic to $I \times \mathbb{S}^{n-2} \times \mathbb{R}$ and two other manifolds that are diffeomorphic to $B^{n-1} \times \mathbb{R}$ and meet along their manifold boundary.

Using the smooth tubular neighborhood theory, we can observe that the union of these manifolds is diffeomorphic to $\mathbb{S}^{n-1} \times \mathbb{R}$.


We concluded that our manifolds are covered by $N\times \bR$ for a closed manifold $N$.	
Note that the possibilities of $N$ are $\SI^{n-1}, \SI^p\times \SI^q$.
Also, $M_K$ for a normal subgroup $K \subset K_h$
of finite index is homeomorphic to $N \times \bR$ in each case. 

In order for there to be a proper action, there exists an element $\gamma$ in 
the deck transformation group of $M_K$ as a cover of $M$
 such that $N \times {0}$ and $\gamma(N \times {0})$ are disjoint.
Let $M'''$ be $M_K/\langle \gamma \rangle$. Then 
$M'''$ admits a fiber bundle $N \ra M''' \ra \SI^1$. 

Let $\Gamma_K$ denote the deck transformation group $M_K \ra M$ isomorphic 
to $\pi_1(M)/K$.
The theory of Benoist in \cite{BenTor} works with the universal cover but the theory also
work with any cover of a holonomy cover as well. 
Lemma 3.1  in \cite{Benoist94}  and the remark after that show that there is 
a connected Lie group acting $\hat N$ on $M_K$ where $\hat N \cap \Gamma_K$ is 
a finite-index normal subgroup of $\Gamma_K$.
This implies that $\gamma^m \in \hat N$ for some $m$,
and $\gamma^m$ is isotopic to the identity map on
$M_K$.

We can then consider an $m$-fold cover $M^{(iv)}$ of $M'''$ corresponding to $\gamma^m$.
We showed above that $M_K$ is diffeomorphic to $N \times \SI^1$, 
and there is a fiber bundle $N \ra M^{(iv)} \ra \SI^1$. 
Since $\gamma^m$ is isotopic to the identity map on $M_K$, 
the monodromy of this bundle is isotopic to the identity and 
hence, $M^{(iv)}$ is diffeomorphic to $N \times \SI^1$. 


	\end{proof} 

\subsection{Reproving Proposition 4.9 of Benoist}\label{subsec:Prop4.9}

We will prove a stronger version of 
Proposition 4.9 of Benoist \cite{BenTor}, which generalizes Nagano-Yagi \cite{NY} and Goldman \cite{Gthesis}, in two steps. We need to do this because the proof in
the original paper is incorrect. Also, the proposition itself only identifies that the manifold
is diffeomorphic to what is written below. 

\begin{thm}[Proposition 4.9 of \cite{BenTor}]\label{thm:main} 
Let $M$ be a closed real projective manifold, $n \geq 3$, 
with an infinite-cyclic holonomy group.
Then $M$ is a pure Hopf manifold, a generalized Hopf manifold or
a pure Benoist manifold. 
Furthermore, $M$ is diffeomorphic to a manifold finitely covered  by
$\SI^1 \times \SI^{n-1}$ and 
$\SI^1\times \SI^p \times \SI^q$, $1 \leq p \leq q \leq  n-2, p+q+1=n$.
\end{thm} 
\begin{proof}
The first part is proved by Proposition \ref{prop:mainp}. 
The last part of the theorem is proved by Proposition \ref{prop:topology} 
\end{proof} 

\begin{rem}\label{rem:BenNotation} 
Proposition 4.9 of \cite{BenTor} states that $M$ is finitely covered by 
$R_g, S_g,$ or $T_{g, p}$ where $g$ is a projective automorphism in the identity component of a Nilpotent Lie group. 
Here, $R_g$ is a Hopf or generalized Hopf manifold, 
$S_g$ is a Benoist manifold containing an elementary one of type III, 
and $T_{g, p}$ is one consisting of elementary ones of type I and II only. 
$S_g$ and $T_{g,p}$ are pure Benoist manifolds. 
\end{rem}


\begin{prop} \label{prop:mainp} 
A closed real projective manifold, $n \geq 3$,  with an infinite-cyclic holonomy group is a pure Hopf manifold, a generalized Hopf manifold, or a pure Benoist manifold. 
\end{prop}  
\begin{proof}
This is proved in Section \ref{subsec:general}. 
In particular, Lemma \ref{lem:persist} gives us the final part. 
\end{proof} 

When $n=2$, the classification due to Goldman \cite{Gthesis} is slightly different 
where the elementary pieces can be assembled in more general ways. 


\section{$\Delta$-orbits} \label{sec:delta}

We will now discuss some consequences of the octantization of Benoist \cite{Benoist94}. 
Let $\Delta$ denote a full positive diagonalizable group over $\bR$. 

The full theory of Benoist in \cite{Benoist94} is very encompassing. 
However, we will concentrate here only on positive diagonalizable groups. Also, in Section 2 of \cite{BenTor}, he gives some summary of theories that are needed here. 
We don't believe we can summarize these better than what he has done already.
If one reduces the theory to diagonalizable groups, this theory is quite straightforward.

\subsection{Simplices} \label{subsec:simplices}
The following material is already explained in Chapter 5 of \cite{BenTor}. 
However, we repeat it here for notational purposes. 

Let $v_1, \dots, v_{n+1}$ dnote the standard basis vectors of $\bR^{n+1}$. 

\begin{itemize} 
\item A subset $I$ of ${\pm 1, \pm 2, \dots, \pm (n+1)}$, where only one of $i$ and $-i$ appears for each integer $i$ such that $1 \leq i \leq n$, is referred to as a {\em  choice set}. 
\item A {\em full choice set} $J$ is a choice set where all absolute values in ${1, 2, \dots, n+1}$ are present. For each $l$, we denote by $J(l)$ the element of $J$ with the absolute value $l$.
\item Given a full choice set $J$ and an integer $i$, we define the {\em lower set} $J_i$ as $J \cap \{\pm 1, \dots, \pm i\}$, and the {\em upper set} $J^i$ as $J \cap \{\pm i, \dots, \pm (n+1)\}$. These lower and upper sets are subsets of $J$ that contain the respective elements.
\item For a choice set $I$, let $\sigma^I$ denote the closed convex simplex with vertices $v_e$ for $e \in I$ on $\SI^n$. 
\item Let $\sigma^{I, o}$ denote the interior of $\sigma^I$. Note that when $I$ has only one or two elements, we may not use the set notation for $I$.
\end{itemize} 
An index set is, of course, a subset of $\{1, 2, \dots, n+1\}$. 


Benoist \cite{BenTor} works with the universal cover; however, the theory applies to any cover of a holonomy cover. 
Proposition \ref{prop:orbitclosure} is Lemma 4.11 in the same paper. 
See Section \ref{subsec:counter} for counterexamples. 
But we restrict ourselves to regular diagonalizable holonomy groups where his proof is correct.

For notational convenience, our holonomy generating 
regular diagonalizable elements are chosen so that its eigenvalues norms
are strictly decreasing. We also provide a short proof since this follows from much general result of \cite{BenTor}.

\begin{prop}[Orbit closures from  Example 5.10 Benoist \cite{BenTor}] \label{prop:orbitclosure} 
Suppose that a closed manifold $W$ has an infinite-cyclic 
regular positive diagonalizable holonomy group with the generator $\gamma$. 
Let $W'$ be a cover of the holonomy cover
with an associated developing map $\dev'$ and 
a holonomy homorphism $h'$. 
Let $\Delta$ denote the unique full positive diagonal group commuting with the holonomy group. 
Let $\mathcal{O}$ be an open $\Delta$-orbit in $W'$. 
Then the closure $\clo_{W'}(\mathcal{O})$ in $W'$ 
is a manifold with corner mapping
diffeomorphic to $Y^J_i:= \sigma^{J_i}  \ast_o \sigma^{J^{i+1}}$ for 
for a full choice set $J$ and some $i$
under $\dev'$ and $\lambda_{i+1}(g) \ne \lambda_i(g)$
for each holonomy element $g$, $g\ne \Idd$. 
\end{prop} 
\begin{proof} 
By Proposition \ref{prop:diagc}, $\Delta = Is_0(W')$, and 
and $\clo_{W'}(\mathcal{O})$ is a union of orbits of $\Delta$.
By Proposition \ref{prop:oct2}, $\dev'| \clo_{W'}(\mathcal{O})$ is 
a diffeomorphism to its image as a cornered manifold. 
We have 
$\pi_1(W) = \langle \gamma, K_W \rangle$ for a deck transformation $\gamma$ going 
to the generator of $h(\pi_1(W))$ and the kernel $K_W$ of the holonomy homomorphism. 
By Proposition \ref{prop:oct2},
$\clo_{W'}(O)$ covers a compact cornered manifold in
a finite cover of $W$ with the deck transformation group $\langle \gamma^m, K_W\rangle$
for some integer $m\ne 0$.  
Hence, $\langle h(\gamma^m)\rangle$ acts on $\dev'(\clo_{W'}(O))$ properly and cocompactly.

Since $\gamma^m$ is regular diagonalizable, $O$ is an open simplex
mapping to an open simplex $\sigma^{J, o}$ in $\SI^n$ with vertices
$v_{j_i}$, $i=1, \dots, n+1$, for some full choice set $J$. 

We claim that $h(\gamma^m)$ acts properly and cocompactly on 
$\sigma^{J_i}  \ast_o \sigma^{J^{i+1}}$ for each $i$: 
Take the subset 
\[S:= \left\{ \left[ \frac{1}{2} v_x + \frac{1}{2} v_y\right] \Bigg| [v_x] \in \sigma^{J_i}, \llrrV{v_x}=1, 
[v_y] \in \sigma^{J^{i+1}}, \llrrV{v_y}=1 \right\}.\] 
Then $S \ast_o h(\gamma^m)(S)$ is a compact fundamental domain
of the action. 

Conversely, we cannot have any smaller union $O'$ of open simplices than 
$\sigma^{J_i}  \ast_o \sigma^{J^{i+1}}$: Otherwise, $O' \cap S \ast_o h(\gamma^m)(S)$ is not 
compact. This violates the cocompactness. 

We cannot have any bigger union $O'$ of open simplices than 
$\sigma^{J_i}  \ast_o \sigma^{J^{i+1}}$:  Otherwise, 
$\bigcup_{i\in \bZ} \gamma^{nm}(S)$ has accumulation points in  
the extra open simplices, contradicting the properness. 

Now, $\dev'$ restricted to $\clo_{W'}(O)$ is a diffeomorphism, 
and the image contains the interior of $\sigma^{J_i}  \ast_o \sigma^{J^{i+1}}$ for some
full choice set $J$ and some $i$. 
Since $\dev'(\clo_{W'}(O))$ is a union of the interior and its sides, 
with a compact quotient,  we must have the result true for some $J$ and $i$.
\end{proof} 

Here, the number $i$ of $Y^J_i$ is called  the {\em type} of $Y^J_i$.
We denote by $\SI^{n-1, j}$ the subset of $\SI^n$ given by setting $x_j =0$. 
We denote by $F^J_{i, j}$ the subset of
$Y^J_i \cap \SI^{n-1, j}$.

Note 
\begin{equation}\label{eqn:face}
F^J_{i, j}=\emp \hbox{ if and only if } i=1, j=1 \hbox{ or } i=n+1, j=n+1.
\end{equation}
If it is not empty, it is of dimension $n-1$ and is 
called the {\em $i$-th face}. 
Also, a $k$-th vertex of $Y^J_i$ is defined as one in
$\{v_k, -v_k\}$ in the closure of $Y^J_i$. 
Note that a $k$-th vertex and a $k$-th face are said to be {\em opposite}.



The following describes the possible change of combinatorics of these sets across the faces.
Loosely speaking, this says that 
the types as denoted by $i$ on $Y^J_i$ 
can change only across the $i$-th face of $Y^J_i$ only by adding
$\pm 1$ to the index $i$. 
This fills in the last part of the proof of Proposition 4.9 of \cite{BenTor}.
\begin{prop}[Next brick \cite{BenTor}]\label{prop:nbrick} 
Assume as in Proposition \ref{prop:orbitclosure}. 
An orbit closure mapping to 
$Y^J_i$ may intersect with another one for $Y^{J'}_j$ in a common face 
if and only if 
one of the following holds\/{\em :} 
\begin{itemize} 
\item $J$ and $J'$ differ by a sign at the $k$-place, $k\ne i, j$,  
we have $i=j$, and  $Y^J_i \cap Y^{J'}_i = F^J_{i, k}$. 
{\em (}Only, $j_k \ra -j_k$ for $k \ne i$)
\item $j= i-1$, $J'_j =  J_{i-1}$,  
and $J^{\prime j+1}  = J^i \cup \{-J(i)\}$.
Furthermore, $Y^J_i \cap Y^{J'}_j = F^J_{i, i}= F^{J'}_{i-1, i} $. 
{\em (}$J(i)$ is deleted from $J_i$ and moved to as $-J(i)$ in $J^{\prime, j}$.{\em )}
\item $j=i+1$, $J'_{j} = J_i \cup \{-J(i+1)\}$,  
and $J^{\prime j+1} = J^i \setminus \{J(i+1)\}$.
Furthermore, $Y^J_i \cap Y^{J'}_j = F^J_{i, i+1}= F^{J'}_{i+1, i+1}$. 
{\em (}$j({i+1})$ is deleted from $J^{i}$, and $-J(i+1)$ added to $J'_{i+1}$.{\em )}
\end{itemize} 
\end{prop}
\begin{proof}
	For the first item,  
suppose that $F^J_{i, k}$ is in the intersecting face of $Y^J_i$ and $Y^{J'}_j$. 
	 Then $F^J_{i, k} = F^{J'}_{j, k}$ and the $k$-th vertices are opposite. 
	 Since $F^J_{i, k}\cap Y^J_i = F^J_{j, k} \cap Y^{J'}_j$, we obtain 
	 \[F^J_{i, k} \cap \sigma^{J_i} = F^J_{j, k} \cap \sigma^{J'_j}
	 \hbox{ and  } F^J_{i, k} \cap \sigma^{J^{i+1}} = F^J_{j, k} \cap \sigma^{J^{\prime j+1}}. 
\]
	 Since a $k$-th face contains all but one vertex, $J$ and $J'$ may differ at only one vertex 
	 $v_k$ or $v_{-k}$. 
	 
Now, for the remaining items, 
note that the intersecting face $F^J_{j, k}$ contains $n$ vertices,
and it is an open join of a simplex in $\sigma^{J_i} \cap \sigma^{J'_j}$ and 
a simplex in $\sigma^{J^{i+1}} \cap \sigma^{J^{\prime j+1}}$. 
Hence, the lower and upper sets of $J$ and $J'$ agree except for one vertex
for $\dim F^J_{j, k}=n-1$, and we have $j = i-1$ or $j= i+1$. 
Hence, we have possibilities for the last two items.
%
\end{proof} 

\subsection{Ambient pieces}\label{subsec:ambient} 


In this section, we will show that given a first lifting of an orbit, there is some set of lifts of orbits 
that are associated with it. This corresponds to the last part of the proof of Proposition 4.9 of 
\cite{BenTor}. 

%



We denote by $\SI^I$ the minimal great sphere containing 
every $\pm v_j$ for $j\in I$ for any choice set  or an index set $I$. 
%
%
We denote by $H^{n-1}_{J(l), J(m)}$ for $l\ne m$ the hemisphere given by $J(m) x_m \geq 0$
in $\SI^{J \setminus \{J(l)\}}$ for a full choice set $J$.


\begin{prop}[Elementary Benoist manifolds of type I]\label{prop:canonicalI}
Assume as in Proposition \ref{prop:orbitclosure}. 
\begin{itemize} 
\item Given $Y^J_1= \sigma^{j_1}\ast_o \sigma^{J^2}$
which lifts to a cover $W'$ of the holonomy cover of $W$, 
there exists a subspace
$\sigma^{j_1}\ast_o H^{n-1}_{j_1, j_2}$, which also lifts to $W'$. 
The manifold boundary is equal to the punctured hemisphere $\sigma^{j_1} \ast_o \SI^{J^3}$.

\item Given $Y^J_{n} = \sigma^{J_n} \ast_o \sigma^{j_{n+1}}$ that lifts to $W'$, 
there exists a subspace $H^{n-1}_{j_{n+1} j_n} \ast_o \sigma^{j_{n+1}}$
that lifts to $W'$. 
The manifold boundary equals the punctured hemisphere 
$\SI^{J_{n-1}} \ast_o \sigma^{j_{n+1}}$. 

\item These are elementary Benoist manifolds of type I. 
\end{itemize} 
\end{prop} 
\begin{proof} 
For the first item, we consider $F^{J'}_{1, k}$ for $k \ne 1, 2$
as $F^{J'}_{1, 1} = \emp$.  We take the union of all $Y^{J'}_1$ that can be obtained 
from $Y^J_1$ by chains of adjacency by the faces of type $F^{J'}_{1, k},  k\ne 1, 2$, 
to obtain $Y^{J'}_1$ for every $J'$ such that $J'_1  = J_1, J'_2= J_2$. 
We never cross over $F^{J'}_{1, 2}$. 
Then the union is of the form in the item. 

We can lift each of the simplices here one by one by using the adjacency
by Proposition \ref{prop:orbitclosure}.
They lift since their union is simply connected
and $W'$ is a union of simplices mapping to the simplices in $\SI^n$ diffeomorphically
since these were lifts of simplices in $\SI^n$. 

For the second part, we consider 
$F^{J'}_{n+1, k}$ for every $J'$ such that $J'_n  = J_n, J'_{n+1}= J_{n+1}$
and $k \ne n+1, n$ as $F^{J}_{n+1, n+1} = \emp$.
We never cross over $F^{J'}_{n, n+1}$. 
Again, the argument is similar. 

The third part is obvious. 
\end{proof}

The first item in the above proposition provides examples of $A_g, d_+ =1$ 
and the second item provides ones of $B_g, d_- =1$  in Example 4.8 of \cite{BenTor}.


We denote by $H^{J_i}$ the hemisphere in $\SI^{J_i}$ given 
by $J(i) x_i \geq 0$, and by $H^{J^{i+1}}$ the hemisphere in $\SI^{J^{i+1}}$ 
given by $J(i+1) x_{i+1}\geq 0$.


The first item in the following proposition provides examples of 
$A_g, d_+ =2$ or $B_g, d_- = 2$, $d_\pm \geq 2$, and 
the second item provides examples of $A_g$ and $B_g$ where 
$d_+ \geq 3, d_- \geq 3$
for $g= h(\gamma)$. 


We denote by $H^n_{J(i)}$ the hemisphere given by $J(i)x_i \geq 0$. 
In the first item of the following $(l, m)$ is the {\em  type} of the canonical ambient piece of type II.
In the second item of the following $i$ is the {\em type} of the canonical ambient piece of type III.

\begin{prop}[Elementary Benoist manifolds of types II, III] \label{prop:canonicalII} 
Assume as in Proposition \ref{prop:orbitclosure}. 
Given $Y^J_i$ that lifts to $W'$, we have the following\/{\em :} 
\begin{itemize} 
\item For $i= 2$, let $(l, m) = (1, 2)$, and 
for $i= n-1$, let $(m, l)=(n, n+1)$.
\[H^n_{J(2)} \setminus H^{J_2} \setminus \SI^{J^3} = 
H^{J_2} \ast_o\SI^{J^3} \hbox{ or } 
H^n_{J(n)} \setminus  \SI^{J_{n-1}} \setminus_o H^{J^n} 
= \SI^{J_{n-1}} \ast_o H^{J^n}\] respectively
contains it, and one of its finite covers lifts to $W'$
as an embedding provided $W'$ is a finite cover of $W_h$. 
{\em (}If $n> 3$, it directly lifts to $W'$.{\em )}
The boundary has two components
that are finite  covers of  two punctured $(n-1)$-hemispheres  
$\{v_{J(l)}\} \ast_o  \SI^{J \setminus \{J(l), J(m)\}}$ 
and $\{v_{-J(l)}\} \ast_o  \SI^{J\setminus \{J(l), J(m)\}}$. 
\item If $2 < i < n-1$, then $Y^J_i$ is in 
$H^{J_i} \ast_o \SI^{J^{i+1}}= H^n_{J(i)} \setminus H^{J_i} \setminus \SI^{J^{i+1}}$ 
and it lifts to $W'$. 
The boundary equals $\SI^{J\setminus \{J(i)\}} \setminus \SI^{J_{i-1}} \setminus \SI^{J^{i+1}}$, which is connected. 
\item The first and second items respectively are elementary Benoist manifolds of types II and III. In the later case, the character is $(i-1, n-i)$. 
\end{itemize} 
\end{prop} 
\begin{proof} 
Suppose that $i=2, n-1$. 
We consider all faces $F^J_{i, k}$ where $k \ne m$. 
Then $F^J_{i, k}$ contains $v_{J(m)}$. 
We can lift each of the simplices here one by one by using the adjacency
using Proposition \ref{prop:orbitclosure}.
By Proposition \ref{prop:nbrick}, we can show that 
we can add a lift of $Y^{J'}_i$ for every possible $J'$ that differs from $J$ except for 
the $m$-th places. This process 
gives us the subspace 
$(\sigma^{J(l), J(m)} \cup  \sigma^{-J(l), J(m)}  ) \ast_o \SI^{J \setminus \{J(l), J(m)\}}$. 

Consider the inverse image $U$ of it in $W'$ or $\tilde W$. 
If $n > 3$, then $|J\setminus \{J(l), J(m)\}| \geq 3$, and the covering must be one-to-one
as the sphere of dimension $>1$ has only a trivial cover. 
Hence, it directly lifts to $W'$ as an embedding. 
For $n=3$,  
a component of $U$ in $W_h$ cannot cover 
$(\sigma^{J(l), J(m)} \cup  \sigma^{-J(l), J(m)}  ) \ast_o \SI^{J \setminus \{J(l), J(m)\}}$ 
infinitely; otherwise, we can obtain a codimension-$0$ submanifold  of $W'$ 
with no compact quotient under $\langle \gamma, K_W\rangle$. 
Hence, we can lift 
a finite cover  of $(\sigma^{J(l), J(m)} \cup  \sigma^{-J(l), J(m)}  ) \ast_o \SI^{J \setminus \{J(l), J(m)\}}$ to
$W'$ as an embedding. 



Suppose that $2 < i < n-1$.  Then $|J_i|, |J^{i+1}| \geq 3$, and 
$\SI^{J_i}$ and $\SI^{J^{i+1}}$ are simply connected. 
Then by taking a union of all subspaces mapping to sets of form $Y^{J'}_i$ 
provided $J'(i) = J(i)$, which can be reached by 
a sequence of adjacent bricks across these faces $F^{J'}_k$, $k \ne i$,
we obtain a subspace mapping diffeomorphic to $H^{J_i} \ast_o \SI^{J^{i+1}}$.
This lifts since the second factor sphere is simply connected. 




The last item is clear. 
\end{proof}

\subsection{Decomposition} \label{subsec:canonical}

In the following, the ``compactness" 
does not have to include the ``Hausdorff property" unlike the French school of 
mathematics. 
The following is essentially proved by Benoist in the proof of Proposition 4.9 in
\cite{BenTor} with more general holonomy assumptions.

\begin{prop}[Benoist] \label{prop:Smillie} 
Assume as in Proposition \ref{prop:orbitclosure}.
Then $W$ decomposes into elementary Benoist manifolds.
%
\qed
\end{prop} 
\begin{proof} 
%
We use Propositions \ref{prop:canonicalI} and \ref{prop:canonicalII}
to show that $W$ splits into elementary Benoist manifolds of various types:
We start with an open orbit. If this is a one giving rise to a type III one, 
then we take the elementary pre-Benoist manifold of type III. Since its boundary
is connected, we can have only one other elementary pre-Benoist manifold of type III adjacent to it. Hence, we obtain a splitting of $W$.
This is a decomposition since the eigenvalues are positive. 

Suppose that we have no open orbit giving rise 
to an elementary Benoist manifold of type III. 
Suppose that there is an open orbit of the type considered in Proposition \ref{prop:canonicalI}. We find the elementary Benoist manifold of type I. 
We consider the adjacent orbit which can be ones for the type I or type II ones. 
We find the elementary Benoist manifold of type I or II with a disjoint interior from
the first one. 
Since the eigenvalues are positive, each elementary Benoist manifold
of type II has two boundary components. 
Then we keep this process to obtain a desired decomposition. 

Finally, if we have only elementary Benoist manifold of type II, 
a similar reasoning will prove the result.

\end{proof} 

For $n=3$, the decomposition is not unique since we are able to use 
$(l. m)=(1, 2)$ or $(l, m)=(3, 4)$ and both will work out.


\subsection{General cases} \label{subsec:general}

%


The purpose is to prove Proposition \ref{prop:mainp}. 
Let $M'$ be a closed real projective manifold with infinite-cyclic $h(\pi_1(M'))$. 
We will now show that 
$M'$ is a pure Hopf manifold, a generalized Hopf manifold, or a pure Benoist manifold.

Let $C_h$ denote the infinite-cyclic holonomy group in $\SL_\pm(n+1, \bR)$
generated by $h(\gamma)$ for the generator $\gamma \in \pi_1(M')$. 
Since there is a factoring $h: \pi_1(M') \ra C_h \ra \SL_\pm(n+1, \bR)$, 
and the inclusion homomorphism of $C_h$ can be deformed to any nearby homomorphisms, 
we can find a one-parameter family of
$h_t:\pi_1(M') \ra \SL_\pm(n+1, \bR)$ so that 
\begin{itemize}
\item $h_t(\pi_1(M'))$ is infinite-cyclic for every $t \in [0, \eps)$, and 
\item $h_{t'_0}(\gamma)$ for some $t'_0\in (0, \eps)$ 
is an semisimple generator and has only eigenvalues that are real or 
complex of moduli in $\bQ \pi$. 
\end{itemize}


Let $(D, h)$ be the element of $\Def''(M')$.  
By the Ehresmann-Thurston principle written in \eqref{eqn:ET}, 
 we can find a parameter $(D_t, h_t)$  
for a developing map $D_t:\tilde M \ra \SI^n$ and
a holonomy homorphism $h_t|\pi_1(M'):  \pi_1(M')\ra \SL_\pm(n+1, \bR)$  
in $\Def''(M')$ for $t \in [0, t'_0]$
so that $D_0 = D$ and $h_0 = h$. 
We denote by $M'_t$ the smooth manifold $M'$ with the corresponding real projective structure. 
Note that $h_t$ has a common kernel to be denoted $K''$.

We can choose $m> 0$ so that 
$h_{t'_0}(\gamma^{\prime m})$ is positive diagonalizable. 
We take the $m$-fold cover $M''$ of $M'$ corresponding to 
$\langle \gamma^{\prime m} \rangle$. 
Now, $\pi_1(M''_{t_0}) = \langle \gamma^{\prime m}, K''\rangle$.
We denote by $M''_t, t\in [0, t'_0]$, $M''$ with the real projective structure 
induced from $M'_t$.  Now we extend the parameter 
to $M''_t, t\in [0, t_0]$ so that $M''_{t_0}$ has a positive regular diagonalizable holonomy
corresponding to $\gamma^{\prime m}$. 
We denote $\gamma'':= \gamma^{\prime m}$.

By Theorem \ref{thm:main}, we have three possibilities: (i) $M''_{t_0}$ is a pure Hopf manifold, 
(ii) a pure generalized Hopf manifold, or (iii) a pure Benoist manifold. 
Now, for $M''_0$, we have: 


\begin{lem}\label{lem:persist}  
We assume that $h(\pi_1(M'))$ is infinite-cyclic. 
Then $M''= M''_0$ 
is a pure Hopf, generalized Hopf, or pure Benoist manifold. 
\end{lem}
\begin{proof} 
We have the one-parameter family 
$(D_t, h''_t)$ where $h''_t= h_t| \pi_1(M'')$. 
In the first two possibilities (i), (ii) of $M''_{t_0}$, 
$D_{t_0}| \tilde M \ra \SI^n$ is a diffeomorphism to 
its image $H^n -\{p\}$ or a covering map to $\SI^n - S^p -S^q$ 
where $H^n$ is a hemisphere and $S^p$ and $S^q$ respectively 
are a great $p$-sphere and a great $q$-sphere disjoint from each other.


In the first case (i), $K''$ is trivial, 
and $\pi_1(M'')$ is generated by an element  $\gamma''\in \pi_1(M'')$ since 
$H^{n, o} \setminus \{p\}$ is a universal cover and 
any group with a nontrivial action on it will have a nontrivial image under $h$. 
Now, $h_{t_0}(\gamma'')$
 has a matrix decomposition corresponding to $p$ and the great sphere 
$S^{n-1} = \partial H^n$ so that we can have a compact quotient space. 
The properness of the action is equivalent 
to the condition that the norm of the eigenvalue corresponding to $p$ is strictly larger or strictly smaller than those associated with $S^{n-1}$ as we can see from 
the projectivizations of the sums of eigenspaces. 
(see Corollaries 1.2.3 and 1.2.4 of \cite{Kt}.) 
This property persists for $t$ in a neighborhood $O$ of $t_0$ for 
some parameter of points $p_t$ and the great spheres $S^{n-1}_t$ bounding a hemisphere 
$H^n_t$ 
where $p_{t_0} = p$ and $S^{n-1}_{t_0} = S^{n-1}$
and $H^n_{t_0}=H^n$. 
Hence, there exists a real projective structure on $M''$ realizing $h''_t$
by taking a developing map a diffeomorphism $\tilde M \ra H^{n, 0}_t \setminus \{p_t\}$. 
By the local homeomorphism property of \eqref{eqn:ET}, 
this structure is unique up to isotopies and 
$D_t:\tilde M \ra \SI^n$ is a diffeomorphism to 
$H^n_t -\{p_t\}$ for $t \in O'$ for a neighborhood $O' \subset O$ of $t_0$. 

Let $A$ denote the subset of $[0, t_0]$ where $(D_t, h''_t)$ is pure Hopf. 
We can show that $A$ is open since we can replace $t_0$ with any $t \in A$.

We claim that this set is also closed. 
Since $A$ is open by the above, we take a component $A_0$ containing $t_0$.
Let $t'$ denote the infimum or the supremum of $A_0$. 
$D_{t'}$ is a limit of $D_t$ for $t \in A_0$ in the $C^r$-topology.
We take the fundamental domain of $M''_{t_0}$ diffeomorphic to 
$F:=\SI^{n-1}\times I$.
Then $D_{t}(F)$ is the fundamental domain in $H^{n,o}_t -\{p_t\}$ for $t\in A$. 
Let $H^{n}_{t'}$ denote one of the geometric limiting hemispheres of $H^n_t$, $t\in A$. 
Let $p_{t'}$ denote one of the limit points of $p_t, t\in A$. 
Since we have the $C^r$-topology finer than the compact open topology,  
$D_{t'}(F)$ again is the fundamental domain in $H^{n,o}_{t'} - \{p_{t'}\}$.
Hence, $A$ is closed. 

This implies $A =[0, t_0]$ and $M''$ is a pure Hopf manifold as well.

Consider now the second case (ii) of generalized Hopf manifolds. 
By Lemma \ref{lem:p=1}, the holonomy cover 
$\tilde M/K''$ is a finite cover of
 $\SI^n \setminus S^p \setminus S^q$
under $D_{t_0}$. 
In addition, $D_t$ induces a finite covering map
$\hat D_t: \tilde M/K''\ra \SI^n \setminus S^p \setminus S^q$
since $K''$ is still the kernel of $h_t|\pi_1(M'')$ by the above construction of $h_t$. 
There is a compact fundamental domain 
$F$ in $\tilde M/K''$ diffeomorphic to $I \times S^p\times S^q$. 
By reasoning similar to the case of pure Hopf manifolds involving
the norms of the eigenvalues and the projectivization of sum eigenspaces,
we can show that $M''_0$ is a generalized Hopf manifold.

Now we go to the pure Benoist-type ones (iii).
Let $A$ denote the subset of $[0, t_0]$ where $(D_t, h''_t)$ is purely Benoist. 


Now, $M''_{t_0}$ is finitely covered by $\SI^1\times \SI^{n-1}$
by the third item of Proposition \ref{prop:topology}. 
First, consider the cases where there are no elementary Benoist manifolds of type III. 

By the Ehresmann-Thurston local homeomorphism principle, 
it suffices to show that for every sufficiently short path in $h''_t$ containing $h''_{t'}$ for some $t'$, we can construct a path of real projective structures with the same type 
decomposition with the corresponding holonomy homomorphism equal to $h''_t$. 

Up to a finite covering, $M''_{t_0}$ is the union of a finite sequence of elementary Benoist manifolds of type II with
two elementary Benoist manifolds of type I at two ends
by Proposition \ref{prop:Smillie}. 
Assume without loss of generality that we have the ones covered
by $\sigma^{j_1}\ast_o H^{n-1}_{j_1j_2}$
up to reversing the choice of the generator of $\pi_1(M''_{t_0})$. 
Then the following two eigenvalue conditions 
are equivalent to the properness of the action
for domains covering the Benoist manifolds of type I and type II. 
\begin{itemize}
\item The associated eigenvalue $\lambda_{1, t_0}$ of
$v_{j_1}$ 
is strictly smaller than $\lambda_{2, t_0}$ of $v_{j_2}$. 
\item The eigenvalues associated with $\gamma$ with $v_{j_1}$ and $v_{j_2}$ are 
strictly less than the norms of the remaining eigenvalues.

\end{itemize} 

Note \eqref{eqn:ET} still holds. 
For $t_1\in A$, 
$h_t(\gamma)$ has the eigenvalue $\lambda_{1, t}$ with the largest norm 
and the second largest norm $\lambda_{2, t}$ of the eigenvalues. 
Since they form continuous functions for $t\in A$, 
we have the following: 
\begin{itemize} 
\item The fixed points $v_{1, t}$ and $v_{2, t}$ corresponding to 
$\lambda_{1, t}$ and $\lambda_{2, t}$ of $h_t(\gamma)$
are well defined, and we can choose these so that 
$t\mapsto v_{1, t}$ and $t\mapsto v_{2, t}$ are continuous. 
\item  $h_t(\gamma)$ acts on a great circle $\SI^1_t$ containing $v_{1, t}$ and $v_{2, t}$ 
and the disjoint great sphere $\SI^{n-2}_t$
which are uniquely determined by eigenvalue conditions. 
\item  $t \mapsto \SI^1_t$ and $t\mapsto \SI^{n-2}_t$
form two continuous functions $A \ra \mathcal{S}^j(\SI^n)$ for $j=1, n-2$, where 
$\mathcal{S}^j(\SI^n)$ is the space of great spheres of dimension $j$. 
\end{itemize} 

In this case, 
we can still obtain the decomposition of Benoist types I and II near the original ones.
Hence, $A$ is open. 
Since  \eqref{eqn:ET} is a local homeomorphism, 
$M''_t$ decomposes into the same number and types of 
elementary Benoist manifolds for $t\in A$ in the identical component of $A$. 

The closedness of $A$ follows since any sequence of totally geodesic subspaces 
decomposing $M'$ will converge to these up to choosing 
subsequences, and two eigenvalue conditions hold by the properness conditions. 
Also, the existence of an elementary Benoist manifold of type I implies that 
$\lambda_{1, t}/\lambda_{2, t} > 1 +\eps$ for $t \in [0, t_0]$ and some $\eps > 0$
since, otherwise, we lose the free action for the elementary Benoist manifold of type I at ends.
 Hence, the number of elementary Benoist type I and II manifolds do not change. 

Finally, suppose that there is an elementary Benoist manifold of type III. 
Since we studied the generalized Hopf manifolds above, we may
assume without loss of generality that
we have a decomposition into two elementary Benoist manifolds of type III
where one of them has the characteristic $(p, q)$ and the other has the
characteristic $(q+1, p-1)$. We can consider the universal cover here as a domain of 
$\SI^n$ by Proposition \ref{prop:propBen}, and $M_h$ is the universal cover. 

Again,  we show that the set of $t$ where the same type of decomposition exists is open and closed. 
For the openness of $A$, 
by the Ehresmann-Thurston local homeomorphism principle, 
it suffices to show that for every sufficiently short path in $h''_t$ containing $h''_{t'}$ for some $t'$, we can construct a path of real projective structures with decomposition of the same type with the corresponding holonomy homomorphism equal to $h''_t$. 

Again, the construction of such a parameter involves
choosing $H^n_t$, $B^p_t\subset \SI^p$, $\SI^q_t$ where $h''_t(\pi_1(M''))$ acts on. 
This can be done in a continuous manner. Hence, the openness follows. 
Adding other generators works just as well. 

For the closedness of $A$, the argument is similar to the above cases. 



\end{proof}

\section{Topological obstructions} \label{sec:obstruct}

The purpose of this section is to prove Corollary \ref{cor:main}. 
First, we do the holonomy version. 
We have an exact sequence 
\[ 1 \ra C \ra h(\pi_1(M)) \ra G \ra 1, C \cong \bZ\]
where a finite group $G$ acts nontrivially on $C$. 

We find a finite regular cover of $M$ by $M'$ so that $h(\pi_1(M'))$ is infinite-cyclic. 
Let $h(\gamma)$ be the generator for $\gamma \in \pi_1(M')$. 
Let $K$ denote the kernel of $h| \pi_1(M')$. 
$K$ is identical to the kernel of $h|\pi_1(M)$ since the latter set does not go to 
a nontrivial element of $G$. 
Hence, the holonomy cover of $M'$ corresponding to $K$
is the same as the holonomy cover $M_h$ of $M$. 

The deck transformation group of $M_h \ra M'$ is infinite-cyclic and is
to be identified with $C$. 
Hence, there is a deck transformation $\delta$ of $\hat M_h \ra M$
such that $\delta t \delta^{-1} = t^{-1}$ for all $t \in C$. 
Finally, let $\dev_h: M_h \ra \SI^n$ denote the developing map. 

Note that we can assume
\[ 1 \ra C \ra \Aut(M_h) \ra G \ra 1\]
where $\Aut(M_h)$ is the deck transformation group of $M_h$
since $h(\pi_1(M))$ is isomorphic to $\Aut(M_h) \cong \pi_1(M)/\ker h$. 

Let $\gamma$ denote a generator of the infinite-cyclic group $C$. 
Now $M_h$ has two ends since $M_h/C$ equals a connected compact manifold $M'$.
Then we can see that there is an end where $\gamma^m$ for sufficiently large $m> 0$ sends 
each end neighborhood into the end neighborhood of the same end. 
We call this the {\em positive} end. The other end is called a {\em negative end}. 


Now, we choose an element $\delta$ of $\pi_1(M)$ going to an element of $G$ 
giving us a nontrivial action on $C \cong \bZ$. 
Since $C$ preserves the ends of $M_h$, 
$\delta$ must switch an end neighborhood of one end of $M_h$ to 
an end neighborhood of the other end by
the twisted extension property of the deck transformation groups of $M_h$. 





\subsection{Those covered by  domains in $\SI^n$.}\label{subsec:domain} 


We will now consider when $M_h$ must be projectively diffeomorphic 
to an open submanifold of $\SI^n$.  
More precisely, we will consider the cases that are forced to be so under the condition on
$p$ and $q$ in Proposition \ref{prop:propBen}. 
$M'$ is  a pure Hopf manifold or a generalized Hopf manifold with $2 \leq p\leq q \leq n-3$
or a union of two elementary Benoist manifolds of type III respectively. 

In a pure Hopf manifold case, 
$M_h= H^{n,0} -\{p\} = \{p\} \ast_o \SI^{n-1}$.
As $n \geq 3$, $M_h$ is a universal cover of $M$ as well. 
We parameterized the space by $\SI^{n-1}\times \bR$. 
Since it is a join, we have directions on each join segment.
We consider $\SI^{n-1}$ to be the positive direction by 
possibly changing the generator of $C$. 
We orient $\bR$ so that it is in the $\SI^{n-1}$-direction. 
Then any end neighborhood $U$ in the positive direction union 
$\SI^{n-1}$ is a neighborhood of $\SI^{n-1}$
Hence, $h(\delta)$ sends $\SI^{n-1}$ to $\{p\}$. However, $h(\delta)$ is a homeomorphism. 
This is a contradiction.


We go to the generalized Hopf manifold cases. 
First, $2 \leq p \leq n-3, 1 \leq q \leq n-3$ implies that 
$M_h$ is projectively diffeomorphic to 
a domain in $\SI^n$, and  $M_h$ is also a universal cover. 
Again, $M$ is covered by 
the simply connected open join $S^p\ast_o S^q = \SI^n \setminus S^p \setminus S^q$
for disjoint great spheres $S^p$ and $S^q$ of dimension 
$p, q$, $2 \leq p \leq n-3, 1 \leq q \leq n-3$, $p+q = n-1$, respectively. 
We consider $S^p$ to be associated with positive direction and 
$S^q$ to be the negative direction. 

By Proposition \ref{prop:topology}, $M$ equals
$S^p\times S^q \times \SI^1$. 
By Corollary 11 of Section 2 of Chapter 7 of \cite{Spanier}, 
$\pi_e(M)$ is isomorphic to $\pi_e(S^p\times S^q \times \SI^1)$ for $e \geq 2$. 
isomorphic to $\pi_e(S^p)\times \pi_e(S^q) \times \pi_e(\SI^1)$. 

Again,  each deck transformation reversing 
the direction sends an end neighborhood of an end
to that of the other end. Hence, $S^p$ is sent to $S^q$ here. 
If $p\ne q$, it is impossible since $\delta$ extends to a projective automorphism of 
$\SI^n$. 
If $p= q$, since $\delta$ must preserve each homotopy class mod $\bZ_2$ of 
$S^p$ and $S^q$ in $S^p \times S^q \times \bR = S^p \ast_o S^q$. 
Since $\delta$ is a projective automorphism, 
we must have $h(\delta)| S^p: S^p \ra S^q$ is a homeomorphism. 
This implies that 
$\delta$ as a deck transformation of $S^p\times S^q \times \bR$
switches the generators of $\pi_p(S^p\times S^q \times \bR)$. 
By Theorem 8 of Section 3 of Chapter 7 of \cite{Spanier}, 
$\delta$ acts to switch the generators of the first two factors 
of $\pi_e(M)$ considered as $\pi_e(S^p)\times \pi_e(S^q) \times \pi_e(\SI^1)$.
This contradicts the premise of Corollary \ref{cor:main}. 


Now, we consider the last case where we have a sum of two elementary Benoist
manifolds of type III with different characteristics according to Proposition \ref{prop:propBen}. 
By Lemma \ref{lem:typeIII} and its proof, 
we can show using their elementary pre-Benoist manifolds 
that $\dev_h$ is a homeomorphism to 
$\SI^n \setminus H^{p}_+ \setminus H^{q+1}_-$ for the type $(p, q)$ for 
an elementary Benoist manifold of type III,
where $\SI^n$ is a union of two hemispheres $H^n_+$ and $H^n_-$, 
and there are two spheres $\SI^p$ and $\SI^{q+1}$ meeting at two points 
$N$ and $N_-$ in the interiors of the two hemispheres respectively,
and $H^p_+ = \SI^p \cap H^n_+$ and $H^{q+1}_- = \SI^{q+1} \cap H^n_-$. 

Again, $h(\delta)$ sends $H^p_+$ to $H^{q+1}_-$ and vice versa.
Hence, $p=q+1$, $p+q=n-1$ and $n=2p$. 

Since $h(\delta^2)$ acts on $H^p_+$, it must fix a point $z$ on it. 
Let $L$ denote the linear map in $\GL(n+1, \bR)$ representing $h(\delta)$. 
Then $\vec x$ so that $x =[\vec x]$ is an eigenvector of $L$. 
We multiply $L$ by a scalar so that the eigenvalue equals $1$. 

Then $[1/2 x + 1/2 Lx]$ is a fixed point of $h(\delta)$ outside 
$H^{p}_+ \cup H^{q+1}_-$, which contradicts the fact that $\delta$ is a deck transformation
and $\dev_h$ is a homeomrphism to its image.
Therefore, we conclude that the possiblities 
involving the elementary Benoist manifolds of
type III also do not happen.

\subsection{General coverings} \label{subsec:generalcov} 

We will now consider the cases in the remaining possibilites for $M$ in
the conclusions of Proposition \ref{prop:propBen}. 
These are generalized Hopf manifolds with $p=1, q=n-2$,
and pure Benoist manifolds without elementary Benoist manifolds of type III.

We will use the notion of Kuiper completions as in \cite{psconv}. 
%
%
Since $M_h$ is a holonomy cover of $M'$, 
the developing map $\dev'$ induces an immersion $\dev_h: M_h \ra \SI^n$ 
equivariant with respect to $h': \pi_1(M')/K' \ra \SL_\pm(n+1, \bR)$.
(See Section \ref{subsec:geostr} for the holonomy covers.)
We can Cauchy complete the path metric induced from the metric on 
$M_h$ induced from the standard Riemannian metric on $\SI^n$. 
We denote the metric by $d$. 
The Cauchy completion is denoted by $\che M_h$. 
The {\em ideal boundary}  $\delta_\infty M_h$ is defined as 
$\che M_h \setminus \hat M_h$.
In addition, each deck transformation extends to 
a self-homeomorphism of $\che M_h$. 
(See \cite{psconv}.)

%


First of all, $M'$ can be a generalized Hopf manifold with $p=1, q=n-2$. 
Here, $M'$ is covered by the universal cover of
$\SI^n \setminus S^1 \setminus S^{n-2}$. 
By Lemma \ref{lem:p=1}, 
the holonomy cover $\hat M_h$ covers $\SI^n \setminus S^1 \setminus S^{n-2}$
finitely by the map $\dev_h$.
Then the completed map
$\overline{\dev}'_h: \che M_h \ra \SI^n$ is a finite branch cover with branch loci $S^1$ and $S^{n-2}$ since we can use the geometry of
tubular neighborhoods of $S^1$ and $S^{n-2}$. 
Since $\pi_1(\SI^n \setminus S^1 \setminus S^{n-2})$ is infinite-cyclic 
or free abelian of rank $2$ when $p=q=1$, 
 $\che M_h\setminus \hat M'$ is a disjoint union of 
a circle $S^1_\infty$ and an $(n-2)$-sphere $S^{n-2}_\infty$, 
which covers $S^1$ and $S^{n-2}$ respectively under $\overline{\dev}'_h$. 

Hence, $M_h$ is diffeomorphic to $S^1 \times S^{n-2}\times \bR$. 
Since $\delta$ extends to a self-homeomorphism of $\hat M_h$, 
it sends $S^1_\infty$ to $S^{n-2}_\infty$ homeomorphically. 
If $n\ne 3$, this is not possible. 
If $n=3$, then $\delta_*$ acts on 
$\pi_1(M_h)\cong \pi_1(S^1_\infty) \times \pi_1(S^{n-2}_\infty)$
switching the two factors, which contradicts the premise
about the action of $\pi_1(M)$ on $\pi_i(M)\otimes \bZ_2$. 


Now, we go to Benoist manifolds with elementary ones of type I. 
$M'$ consisting of only two types 
$(1)$ and $(1, 2)$ or $(n+1)$ and $(n, n+1)$. 
The universal cover $\tilde M$ decomposes into 
submanifolds $N_i$ real projectively diffeomorphic to 
the covers of defining covers of elementary Benoist manifolds:
\begin{itemize}
\item $Q^o \cup (H^{n-1, o}\setminus  \{p\})$ for an $n$-bihedron $Q$ 
bounded by two $(n-1)$-hemispheres, one of which is $H^{n-1}$ and $p \in H^{n-1, o}$. 
\item finite covers of $H^{n} \setminus S^{n-2} \setminus S^1$ for an $n$-hemisphere $H^n$ and  disjoint great spheres 
$S^{n-2} \subset \partial H^n$ and a great circle $S^1$ passing $H^{n, o}$.
(See the proof of Proposition \ref{prop:mainp}.)
\end{itemize} 
In our case, since two elementary Benoist manifolds at the ends are of the first type, 
we see that the universal cover 
$\tilde M$ can be constructed by simply gluing projective 
copies of these along the 
punctured $(n-1)$-dimensional hemispheres. 
Hence, there are  submanifolds 
$N_1, \dots, N_m$ embedded in $\tilde M$
developing under $\dev'$ to the above manifolds in the two items.  
Also, we only need the domain itself for the second item 
since the first item ones are at the end of the finite sequence, and these 
have to match up. 

The closure of $N_i$ in $\che M$ under $\overline{\dev}'_h$ is homeomorphically mapped to
an $n$-bihedron or an $n$-hemisphere: This is because
the closure is the Cauchy completion, 
and the Cauchy completions are unique up to isometry.
(See Section 3 of  \cite{psconv}.) 
Hence, $\che M_h$ is a union of $\che N_1, \dots, \che N_m$ for some $m \geq 2$.

We may identify $\che N_i$ by $\overline{\dev}'_h$ to 
copies of $Q$ or $H^n$ in $\SI^n$ since these are the Cauchy completions. 
The above covers, respectively, give us the ideal boundary sets as follows:
\begin{itemize} 
\item The disjoint union of $\{p\}$ and the closed hemisphere
in $\partial Q$ other than $H^n$. 
\item The disjoint union of $S^{n-2} \subset \partial H^n$ 
and a segment $S^1 \cap H^n$. 
\end{itemize}
Since the ideal boundary of $\che M_h$  is a union of the set of ideal points in $\che N_i$, 
the ideal boundary is a disjoint union of 
\begin{itemize} 
\item an $(n-1)$-sphere that is 
the  union of two $(n-1)$-hemispheres in the ideal boundary of 
a bihedron for an  elementary Benoist manifold of type I, and 
\item a  segment that is a finite union of 
the segments that correspond to the intersections of great circles with $n$-hemispheres. 
\end{itemize} 
Since the set of ideal points has two connected components, 
the deck transformation $\delta$ extends to a self-homeomorphism of $\che M_h$ sending the sphere of dimension $n-1> 2$  to a segment. (For an explanation, see Chapter 3 of \cite{psconv}.) 
This is a contradiction. 

We conclude 
that $h(\pi_1(M))$ cannot be a nontrivially twisted finite extension of $\bZ$.

\begin{proof}[Completion of the proof of Corollary \ref{cor:main}]  
Now we go to the fundamental group part of the corollary; that is, 
we will show that $\pi_1(M)$ cannot be the nontrivially twisted 
extension of infinite-cyclic $C$. Let $t$ be a generator of $C$ and $\delta$ be an element so that $\delta t \delta^{-1} = t^{-1}$. 
If $h(t)$ is of finite order, then $h(\pi_1(M))$ is finite. 

If $h(\delta)$ is trivial, then $h(t) = h(t)^{-1}$, and $h(\pi_1(M))$ is finite. 
A closed real projective manifold $M$ with a finite holonomy group must have 
a spherical metric since the holonomy preserves some 
spherical metric on $\SI^n$. 
Hence, $\pi_1(M)$ is also finite.
This is absurd; $h(\delta)$ is not trivial, and $h(t)$ is of infinite order. 

Hence, $h(\pi_1(M))$ is a nontrivially twisted extension of $\bZ$. 
Now we apply the above result for the holonomy group. 
\end{proof}

\subsection{Some examples} \label{subsec:counter} 


Finally, we mention the following minor point. 
Lemma 4.11 of \cite{BenTor} is incorrect if we do not assume 
that the holonomy group action is diagonalizable. 
There $Y^+(i)$ is defined as follows: 
Suppose that a projective automorphism $g$ has positive eigenvalues 
$\alpha_1 > \alpha_2 > \dots > \alpha_{r} > 0$. 
Let $E_i = \ker (g - \alpha_i\Idd)$ and $d_i = \dim E_i$. 
$E = E_1 \oplus \cdots \oplus E_{r}$.
Let $N$ be a maximal connected nilpotent group containing $g$, 
and $\mathcal{O}$ be an $N$-orbit in $\SI^n$. 
For $i=1, \dots, r$, we define $Y^+(i)$ as
\[\overline{\mathcal{O}} \setminus \Pi(E_1 \oplus \dots \oplus E_i -\{O\}) \setminus 
 \Pi(E_{i+1} \oplus \dots \oplus E_r -\{O\}).\]


An open hemisphere in $\SI^n$ has a natural structure of the complete affine space. 
We will consider a complete affine space as an open hemisphere. 
Take a complete affine line $L$ on $\SI^1$ where a discrete translation group $\langle g\rangle$ acts on. 
The quotient is a projective $1$-manifold.
This is not in the form of $Y^+(i)$ immediately above Lemma 4.11 of 
\cite{BenTor} since 
one is removing the sums of cyclic subspaces associated with eigenvalues less than or greater than some number in that paper. 

We can make this example into higher dimensional ones. 
Let $\SI^n$ denote the real projective sphere and let 
$[e_i]$, $i=1, \dots, n+1$, the directions of the standard frame vectors.
We let $L$ be the complete affine line on
 the great circle that is a span of $[e_1], [e_2], [-e_1]$
with endpoints $[e_1], [-e_1]$.
We take $g$ as follows: 
\newcommand{\bigzero}{\mbox{\normalfont\Large\bfseries 0}}
\renewcommand{\arraystretch}{1.5}
\[
\left(\begin{array}{@{}c|c@{}}
  \begin{matrix}
  \lambda_1& 1 \\
  0 & \lambda_1
  \end{matrix}
  & \bigzero \\
\hline
  \bigzero &
  \begin{matrix}
  D(\lambda_3, \dots, \lambda_{n+1}) 
  \end{matrix}
\end{array}\right)
\]
where 
$D(\lambda_3, \dots, \lambda_{n+1})$ is a diagonal $(n-2)\times(n-2)$-matrix
with diagonal entries $\lambda_3, \dots, \lambda_{n+1}$, 
$\lambda_1 = \lambda_2 > \lambda_3, \dots, \geq \lambda_{n+1} > 0, 
\lambda_1 \lambda_2 \dots \lambda_{n+1} = 1.$
We define a simplex $\sigma :=[e_3] \ast \cdots \ast [e_{n+1}]$.
We let
$X : = \clo(L) \ast \sigma - \{[e_1]\}  - [-e_1]\ast \sigma$
where $g$ acts on. 
One can think of $X$ as the set of points represented 
by 
$a_1 e_1 + a_2 e_2 + \sum_{i=3}^{n+1} a_i e_i$ 
where $a_i \geq 0$ for $i\geq 2$ and  
we cannot have $a_1< 0$ and $a_2 =0$,  
and cannot have $a_1> 0$ and $a_i =0, i \geq 2$. 

Let $||\cdot||$ denote the standard Euclidean norm in $\bR^{n+1}$ containing $\SI^n$. 
We define $\sigma_{1/2}$ as the set $\{[e_1/2 + x/2]| ||x||=1, [x]\in \sigma\}$. 
Then we take the join $\hat \sigma:= [e_2] \ast \sigma_{1/2}$
which is the intersection of $X$ with a hyperspace containing $[e_2]$ and 
$\sigma_{1/2}$. Hence, $\hat \sigma$ separates $[e_1]$ from
$[-e_1]\ast\sigma$ in $X$. 
Now, $g$ sends it to a disjoint subspace
since $\sigma_{1/2}$ moves closer to $[e_1]$ strictly and 
so does $[e_2]$.
Hence, $g(\sigma_{1/2})$ is in one of the two components of 
$X - \sigma_{1/2}$ closer to $[e_1]$,
and $\hat \sigma$ and $g(\hat \sigma)$ bound a compact domain $J$ in $X$. 
Then $g^i(J)$ is disjoint from $J$ or meets $J$ at $\hat \sigma$ or $g(\hat \sigma)$. 
Also, $\bigcup_{i\in \bZ} g(J) =X$ since we can understand the positions of 
$g^i(\hat \sigma)$ by $g^i([e_2])$ and $g^i(\sigma_{1/2})$. 
Hence, $J$ is the fundamental domain of the action of $\langle g \rangle$. 
Again, $X$ cannot be of the form $Y^+(i)$. 

We can also extend this example such that $g$ has a Jordan decomposition of 
block sized $2, j_1,\dots, j_p$
replacing $\sigma$ with $S^{j_1}_1 \times \dots \times S^{j_p}_p$ for some collection of independent sets $S^{j_k}_k$ for $k=1, \dots, p$, that are either 
a great sphere of dimension $j_k$ if $j_k > 0$ or a point if $j_k=0$
corresponding to the blocks. 
The norms of eigenvalues are chosen to be all less than $\lambda_1$. 

By Proposition  \ref{prop:mainp}, 
our examples do not occur when the holonomy groups are virtually
infinite-cyclic for closed manifolds
since they are not elementary Benoist manifolds nor can they be decomposed 
into ones.  However, it is still unknown if this could 
occur as an embedded submanifold inside a closed real projective manifold 
of some complexity.  We leave this as a question. 
(See also \cite{ChoiCCD}.)

\begin{rem} 
It is unlikely that we can generalize this example further.
We tried increasing the multiplicity to $\lambda_1$ and adding 
more eigenvalues above $\lambda_1$. The compact Hausdorff quotients do not
seem to exist here. 
\end{rem} 




As a final example, we mention that 
\[(\SI^p \times \SI^p \times \SI^1)/\sim \hbox{ where }
(x, y, t) \sim (-y, x, 1-t)\]
admits a real projective structure for $n= 2p+1$. 
We construct this from 
\begin{multline}
\SI^{2p+1} \setminus \SI^{\{1, 2,\dots, p+1\}} \setminus \SI^{\{p+2, p+3, \dots, 2p+2\}} 
\cong  \SI^{\{1, 2,\dots, p+1\}} \ast_o \SI^{\{p+2, p+3, \dots, 2p+2\}}
\end{multline} 
and taking the quotient by the group generated by 
$g$ given by a diagonal matrix with eigenvalues $\lambda> 1$ for 
the first $p+1$ basis vectors and $1/\lambda$ for the remaining basis vectors
and $\delta$ is given by the projectivization of 
the map $\hat \delta: \bR^{p+1}\times \bR^{p+1}$ 
given by 
\[t(x, 0) + (1-t)(0, y) \mapsto (1-t)(-y, 0)+ t(0,x) \hbox{ for }x, y \in \bR^{p+1}.\] 

These belong to exceptional ones in Theorem 1.1(ii)of \c{C}oban \cite{Coban21}.

\bibliographystyle{plain} 

\bibliography{aff3bib.bib} 

\end{document}